\documentclass[a4paper,fleqn]{cas-sc}

\usepackage[numbers,compress]{natbib}
\usepackage{graphicx}    % 使用 graphicx 而不是 graphics
\usepackage{subcaption}  % 使用 subcaption 而不是 subfigure
\usepackage{wrapfig}
\usepackage{caption}
\graphicspath{{eps/}}
\usepackage{epstopdf}    % 通常与 graphicx 宏包一起使用
\usepackage{times}       % 建议使用 newer package for new fonts scheme 宏包替代
\usepackage{amsmath}
\usepackage{amssymb}
\usepackage{amsfonts}
\usepackage{fancyhdr}
\usepackage{color}
\usepackage{enumerate}
 \usepackage{hyperref}    % 放在最后加载，通常建议放在导言区的最后加载
 
\DeclareMathOperator{\esssup}{ess\,sup}
\newtheorem{theorem}{Theorem} %
\newtheorem{lemma}{Lemma}
\newtheorem{corollary}{Corollary}
\newtheorem{proposition}{Proposition}

\newtheorem{definition}{Definition}

\newtheorem{remark}{Remark}
\newenvironment{proof}{{\it Proof:\enspace}}{\hfill $\blacksquare$\par}
\def\tsc#1{\csdef{#1}{\textsc{\lowercase{#1}}\xspace}}
\tsc{WGM}
\tsc{QE}
\tsc{EP}
\tsc{PMS}
\tsc{BEC}
\tsc{DE}
\begin{document}
%%\let\WriteBookmarks\relax
%%\def\floatpagepagefraction{1}
%%\def\textpagefraction{.001}
%%\shorttitle{Leveraging social media news}
%%%\shortauthors{J.K. Krishnan et~al.}
%\begin{frontmatter}
\let\WriteBookmarks\relax
\def\floatpagepagefraction{1}
\def\textpagefraction{.001}

\shorttitle{FTISS for infinite-dimensional systems}    

\shortauthors{Sun et al.}  
\title [mode = title]{Finite-time input-to-state stability for infinite-dimensional systems}                      
%%\tnotemark[1,2]

%%\tnotetext[1]{This document is the results of the research
%%   project funded by the National Science Foundation.}

%%\tnotetext[2]{The second title footnote which is a longer text matter
%%   to fill through the whole text width and overflow into
%%   another line in the footnotes area of the first page.}

\author[1]{Xiaorong Sun}%[<options>]

%\cormark[1]

%%\fnmark[1]

\ead{Xiaorong.Sun@my.swjtu.edu.cn}

%%\ead[url]{}

%%\credit{}

\affiliation[1]{organization={School of Mathematics},
            addressline={Southwest Jiaotong University}, 
            city={Chengdu},
%          citysep={}, % Uncomment if no comma needed between city and postcode
            postcode={611756}, 
            state={Sichuan},
            country={China}}

\author[1,2]{Jun Zheng}%[]

%%\cormark[1]
% Footnote of the second author
%%\fnmark[1,2]

% Email id of the second author
\ead{zhengjun2014@aliyun.com}

%%\ead[url]{}

%%\credit{}
\author[2]{Guchuan Zhu}
%%%\cormark[1]
% Footnote of the second author
%%\fnmark[2]

% Email id of the second author
\ead{guchuan.zhu@polymtl.ca}

%%%\ead[url]{}

%%%\credit{}

\affiliation[2]{organization={Department of Electrical Engineering},
           addressline={Polytechnique Montr\'{e}al}, 
                       city={ Montreal},
        % citysep={}, % Uncomment if no comma needed between city and postcode
                     postcode={H3T 1J4}, 
           state={Quebec},
           country={Canada}}

%%\cortext[1]{Corresponding author}

%%\fntext[1]{}


\begin{abstract}
In this paper, we extend    the notion of finite-time input-to-state stability (FTISS) for finite-dimensional systems to infinite-dimensional systems. More specifically, we first prove   an FTISS Lyapunov theorem for a class of infinite-dimensional systems,  namely, the existence of an FTISS Lyapunov functional (FTISS-LF) implies the FTISS of the system, and then,  provide a sufficient condition for ensuring the existence of an FTISS-LF for a class of abstract infinite-dimensional systems under the framework of compact semigroup theory and Hilbert spaces. As an application of the FTISS Lyapunov theorem,  we verify the FTISS   for a class of parabolic PDEs involving  sublinear terms and distributed in-domain disturbances. Since the nonlinear terms of the corresponding abstract system  are not Lipschitz continuous,  the well-posedness is proved based on  the application of compact semigroup theory  and the  FTISS is assessed  by using the Lyapunov method with the aid of an interpolation inequality. Numerical simulations are conducted to  confirm the   theoretical results.
\end{abstract}

%%\includegraphics{figs/cas-grabs.pdf}

%%\item Research highlights item 1
%%\item Research highlights item 2
%%\item Research highlights item 3


\maketitle

\section{Introduction}

Originally introduced by Sontag in 1989 \cite{Sontag:1989}, the notion of input-to-state stability (ISS) provides a powerful tool for  characterizing the influence of external inputs on the stability of
 finite-dimensional systems.
 The ISS theory  becomes rapidly  one of the pillars in the  nonlinear and robust control \cite{Khalil:1996, Mironchenko:2023, Sontag:1995, Sontag:1996}   and has  a wide range of applications in various fields, e.g.,  robotics \cite{Park:2020}, aerospace engineering\cite{Wang:2020a}, transportation\cite{Wang:2020b},  etc.
  Roughly speaking, if a system is ISS, then it is asymptotically stable in the absence of external inputs while  keeping certain  robust properties,  such as ``bounded-input-bounded-state'',    in the presence of external inputs. Especially,   the state  of the system should be eventually small when the inputs are small.

 Extending the    ISS theory of  finite-dimensional systems   to infinite-dimensional systems started around 2010 \cite{Dashkovskiy:2010a, Jayawardhana:2008} and has achieved significant progress in the past decade; see, e.g.,  \cite{Damak:2021, Damak:2023, Dashkovskiy:2013, Jacob:2018, Jacob:2020,  Mironchenko:2020, Mironchenko:2018} for ISS-Lyapunov characterizations   for abstract infinite-dimensional systems;  \cite{Jayawardhana:2008, Karafyllis:2016, Karafyllis:2017, Karafyllis:2019, Lhachemi:2019, Mironchenko:2015, Orlov:2020, Zheng:2018, Zheng:2019, Zheng:2020b} for the ISS assessment of  partial differential equations (PDEs) with different types of disturbances; \cite{Karafyllis:2016, Zheng:2019, Mironchenko:2019b, Smyshlyaev:2004, Wang:2022, Zhang:2021} for the input-to-state stabilization of PDEs under backstepping control, and \cite{Aguilar:2021, Edalatzadeh:2019, Zhang:2024a} for  the application of ISS  to
PDEs arising in  multi-agent control, the railway track model, and power tracking control,  just to cite a few.

 It is worth  mentioning that in \cite{Hong:2008, Hong:2010} the authors introduced a new stability concept, which is stronger than the ISS,   to tackle finite-time control problems (see \cite{Bhat:1998, Coron:1995, Hong:2001, Hong:2002, Moulay:2008a, Oza:2013, Wang:2009, Wang:2012}) for finite-dimensional nonlinear systems with uncertainties, namely,  the  finite-time input-to-state stability (FTISS). More specifically, {taking into account the properties of FTS and ISS,} the FTISS of a system requires    that  in the absence of external inputs   the state of the   system should reach equilibrium within a finite time,  while   in the presence of external inputs the state can reach a given bounded region in finite time \cite{Hong:2008, Hong:2010,Lopez-Ramirez:2018, Lopez-Ramirez:2020}. Moreover, the state   should be   small when the inputs are small.
     Therefore, the notion of FTISS  provides a refined characterizations for the robust stability of finite-dimensional  systems, which  plays a key role in the study of finite-time  stability and stabilization of finite-dimensional nonlinear systems (see \cite{Hong:2008, Hong:2010, Aleksandrov:2022, Liang:2023})  and has attracted much attention in the past few years  \cite{Lopez-Ramirez:2018, Lopez-Ramirez:2020, Liang:2023, Li:2018, Zhang:2024b}.  Especially, the {FTISS Lyapunov theorem}, which states that the existence of an FTISS Lyapunov functional (FTISS-LF) implies the FTISS of the system, has been proved for certain finite-dimensional   systems \cite{Hong:2008, Hong:2010, Lopez-Ramirez:2018, Lopez-Ramirez:2020, Aleksandrov:2022, Liang:2023}.
     % ,  and sufficient conditions for the existence of an FTISS-LF has  been provided \cite{Hong:2008, Hong:2010,Lopez-Ramirez:2018, Lopez-Ramirez:2020}.

     {The first attempt to extend the concept of  FTISS to infinite-dimensional systems is due to  \cite{Sun:2024}, where, as a special case of FTISS, the notion of prescribed-time input-to-state stability (PTISS) was extended to infinite-dimensional systems. Moreover, a {PTISS Lyapunov} theorem was proved and a sufficient condition  for the existence of a PTISS Lyapunov functional was provided for a class of infinite-dimensional systems under the framework of Hilbert spaces \cite{Sun:2024}. Unlike the FTISS, for which the settling time may  depend  on the initial data and be unknown in advance, the PTISS indicates that the system can be stabilized within a  prescribed finite time, regardless of its initial data. Especially, as  addressing the PTISS of finite-dimensional systems {\cite{Song:2017}}, by introducing a monotonically increasing function  $\vartheta(t):=\frac{T}{T-t}$ with a prescribed finite time $T$  to the structural conditions of FTISS-LFs, it is straightforward to prove the PTISS for a class of infinite-dimensional systems  or parabolic PDEs with time-varying reaction coefficients having a form of $\vartheta(t)$   by using the Lyapunov method \cite{Sun:2024}.}

     % However,  {compared with \cite{Sun:2024}}, for infinite-dimensional systems,   {studying the FTISS in a generic case,     {where the  finite time  is unknown in advance and may be dependent of initial data}, namely,  {the FTISS with     non-prescribed finite time}, is more challenging and no relevant results have yet   been reported in the existing literature.

       {It is worth noting that for infinite-dimensional systems, studying the FTISS in a generic case where the stabilization time is unknown in advance and may be dependent of initial data, i.e., the FTISS without prescribing the settling time, is indeed more challenging compared to the case of PTISS, and no relevant results have yet been reported in the existing literature.}
       The  main obstacle in addressing   the FTISS   for infinite-dimensional systems  may lie in verifying   sufficient conditions for the existence of an  FTISS-LF. In particular, for specific PDEs, it is  difficult to validate the effectiveness of a Lyapunov candidate in the FTISS analysis {due to the fact that sublinear terms are usually involved and cannot be easily handled.}  {In addition, compared to   the PTISS analysis of infinite-dimensional systems, for which  strongly continuous semigroup ($C_0$-semigroup) {generated by bounded linear operators} is often used to ensure the well-posedness, as shown in  Section~\ref{4} and \ref{g} of this paper, even for parabolic PDEs,  additional   properties of $C_0$-semigroup are needed  for proving the well-posedness of the corresponding abstract systems due to the appearance of non-Lipschitz continuous terms  when the FTISS   is considered. This also represents a challenge.}

%Note that a  robust stability property between the ISS and the FTISS was investigated  in \cite{Han:2024} for PDEs under boundary finite-time control. More precisely, for the  heat equation, a boundary controller was designed to ensure that
%the  closed-loop system is finite-time   stable in the absence of disturbances, while  it remains ISS in the presence of distributed in-domain disturbances \cite{Han:2024}. Nevertheless, since it is  challenging to choose a suitable Lyapunov candidate and  verify the FTISS for a PDE, the FTISS   of the  heat equation in closed loop has not been proven in \cite{Han:2024}.

The aim  of this work is to {study the FTISS  for   infinite-dimensional systems  without prescribing the settling time} and  provide   tools for establishing  the FTISS    for certain nonlinear infinite-dimensional systems. In particular, as a first attempt in addressing the FTISS    for PDEs, we  show how to {verify the well-posedness based on the application of the compact semigroup theory and to} use the interpolation inequality  to overcome the difficulties in verifying the structural conditions of Lyapunov functionals for a class of   parabolic PDEs with  {sublinear terms} and distributed in-domain disturbances. Overall, the main contribution  of this work include:
\begin{enumerate}
\item[(i)]  extending the notion of FTISS for finite-dimensional systems to infinite-dimensional systems and proving a Lyapunov theorem, which states that the existence of an  FTISS-LF  implies the FTISS of the system;
\item[(ii)]  providing a sufficient condition  to guarantee the existence of an FTISS-LF for  certain nonlinear  infinite-dimensional systems {under the framework of compact semigroup theory} and Hilbert spaces, thereby providing   tools for   stability analysis of infinite-dimensional systems;
    \item[(iii)]  proving an interpolation inequality, which paves the way to  assess the FTISS   for PDEs, and verifying the sufficient condition for the existence of an FTISS-LF for   a class of parabolic PDEs with {sublinear terms} and distributed in-domain disturbances.
\end{enumerate}

%  In the rest of the paper, we  introduce first some basic notations used in this paper.
%  In Section~\ref{3.1} and \ref{3.2}, we introduce the notions of FTISS and FTISS-LF and prove the FTISS Lyapunov theorem for   infinite-dimensional  systems under a general form, respectively. In Section~\ref{4}, we provide a sufficient condition that ensures the existence of an FTISS-LF for certain  infinite-dimensional nonlinear systems {under the framework of compact semigroup theory} and Hilbert spaces. In Section~\ref{g}, considering the application of FTISS Lyapunov theorem, we first prove an important interpolation inequality, and then use it to verify the FTISS   for a class of parabolic PDEs with {sublinear terms} and distributed in-domain disturbances. We also conduct numerical simulations to illustrate the obtained theoretical results.  Finally, some conclusion remarks are given in Section \ref{s6}.

    In the rest of the paper, we  introduce first some basic notations used in this paper.
  In Section~\ref{3.1} and \ref{3.2}, we introduce the notions of FTISS and FTISS-LF and prove the FTISS Lyapunov theorem for   infinite-dimensional  systems under a general form, respectively. In Section~\ref{4}, we provide a sufficient condition that ensures the existence of an FTISS-LF for certain  infinite-dimensional nonlinear systems  under the framework of compact semigroup theory  and Hilbert spaces. In Section~\ref{g}, considering the application of FTISS Lyapunov theorem, we verify the FTISS for a class of parabolic PDEs with  sublinear terms  and distributed in-domain disturbances. More specifically, we first prove the well-posedness by using  the compact semigroup theory in Section~\ref{Well-posedness}. Then, we prove an interpolation inequality and  use it to verify the FTISS   for the considered PDEs in Section~\ref{Sec: proof of Prop.}. We also conduct numerical simulations to illustrate the obtained theoretical results in Section~\ref{Simulation}.  Finally, some conclusion remarks are given in Section \ref{s6}.
\paragraph*{Notation} Let $\mathbb{R}:=(-\infty,+\infty)$, $\mathbb{R}_{>0}:=(0,+\infty)$ and $\mathbb{R}_{\geq 0}:=\mathbb{R}_{>0}\cup \{0\}$.

 For $p\in[1,+\infty)$, the space $L^p(0,1)$  consists of $p$-th power integral functions $g: (0,1)\to\mathbb{R}$ satisfying $ \int_0^1\left|g(x)\right|^p\mathrm{ d}x <+\infty$ and is endowed with norm $\|g\|_{L^p(0,1)}:=\left(\int_0^1\left|g(x)\right|^p\mathrm{ d}x\right)^{\frac{1}{p}}$.
The space $L^{\infty}(0,1)$ consists of measurable functions $g: (0,1)\to\mathbb{R}$ satisfying $\mathop{\esssup}\limits_{x\in(0,1)}\left|g(x)\right|<+\infty$ and is endowed with the norm $\|g\|_{L^{\infty}(0,1)}:= \mathop{\esssup}\limits_{x\in(0,1)}\left|g(x)\right|$.

For a positive integer $k$ and a constant $ p \in [1,+\infty)$, the Sobolev space $W^{k,p}(0,1)$ consists of functions belonging to $  L^p(0,1) $ and having weak derivatives of order up to $k$, all of which also belong to $L^p(0,1)$. The norm of a function  $g\in W^{k,p}(0,1)$ is defined by $\|g\|_{W^{k,p}(0,1)}:=\left(\int_0^1\sum\limits_{i=0}^k\left|\frac{\partial^{i}g(x)}{\partial x^i}\right|^p\mathrm{ d}x\right)^{\frac{1}{p}}$. Let $H^2(0,1):=W^{2,2}(0,1)$ and  $W_{[0]}^{1,p}(0,1):=\left\{g\in W^{1,p}(0,1)\left|\right.~g(0)=0\right\}$.

Denoted by $C\left(\mathbb{R}_{\geq 0};\mathbb{R}_{\geq 0}\right)$   the set of all continuous functions $g: \mathbb{R}_{\geq 0}\to \mathbb{R}_{\geq 0}$.
For a  normed linear  space $Y$, $BC\left( \mathbb{R}_{\geq 0};Y\right)$ denotes the set of all continuous functionals $g:\mathbb{R}_{\geq 0}\to Y$ with {$\|g\|_{BC\left( \mathbb{R}_{\geq 0};Y\right)}:=\sup\limits_{s\in\mathbb{R}_{\geq 0}}\|g(s)\|_{Y}<+\infty$.}

For     normed linear spaces $X$ and $ Y$, let $L(X,Y)$ be the space of bounded linear operators $P:X\to Y$. Let $L(X):=L(X,X)$  with the norm
    \begin{align*}
 \|P\| :=\|P\|_{L(X)}:=\sup\left\{\|Px\|_{X} \left|\right.~x\in X, \|x\|_{X}\leq 1\right\}.
 \end{align*}

 For a given operator $A$, $D(A)$ denotes the domain of $A$ and $\rho(A)$ denotes the resolvent set of $A$. Let $R(\lambda:A):=(\lambda I-A)^{-1}$, where $\lambda$ is a complex number and $I$ represents the identity operator on $D(A)$.

 For different classes  of comparison functions {\cite[Appendix A.1, p.307]{Mironchenko:2023}}, let
\begin{align*}
\mathcal{P}:=&\left\{\gamma\in C\left(\mathbb{R}_{\geq 0};\mathbb{R}_{\geq 0}\right)\left|\right.~\gamma(0)=0\ \text{and}\ \gamma(s)>0 \ \text{for} \ \text{all}\ \  s\in \mathbb{R}_{>0}\right\},\\
\mathcal{K}:= & \left\{ \gamma \in \mathcal{P}\left|\right.~\gamma\ \text{is strictly increasing}\right\},\notag\\
\mathcal{K_\infty}:=& \left\{\gamma\in\mathcal{K}\left|\right.~\gamma\ \text{is\ unbounded}\right\}.
%\mathcal{L}= &\{\gamma \in C(\mathbb{R}_{>0};\mathbb{R}_{>0}) : \gamma\ \text{is strictly decreasing}\\
          %     &\text{with} \lim\limits_{t\to +\infty}\gamma(t)=0\},\\
%\mathcal{KL}= &\{\gamma\in C(\mathbb{R}_{>0}\times\mathbb{R}_{>0};\mathbb{R}_{>0}) : \gamma(\cdot, t)\in\mathcal{K},\\
   %   &\forall t\geq 0; \gamma(r,\cdot)\in\mathcal{L}, \forall r>0\}.
\end{align*}

Denote by   $f\circ g$  the composition of the functions $f$ and $g$, i.e., $f\circ g (\cdot):=f(g(\cdot))$.

For a given function $g:(0,1)\times\mathbb{R}_{\geq 0}\rightarrow\mathbb{R}$, we use the notation $g[t]$ to denote the profile at certain $t\in\mathbb{R}_{\geq 0}$, i.e., $g[t](x)=g(x,t)$ for all $x\in(0,1)$.

\section{FTISS for   infinite-dimensional  systems}\label{3}
In this section, we present  the notion of  FTISS  and an FTISS Lyapunov theorem  for a class of infinite-dimensional systems, which can be generated by PDEs, abstract differential equations in Banach spaces,  time-delay systems,  etc.

\subsection{The notion of FTISS for infinite-dimensional systems}\label{3.1}
We first recall the notion of a  control system, defined below, which comprises ODE and PDE control systems   as  special cases.
\begin{definition}{\cite[Definition~6.1, p.~239]{Mironchenko:2023}}\label{S}
Let the triple $\Sigma=\left(X,U_c,\phi\right)$  consist of the Banach spaces $\left(X,\|\cdot\|_X\right)$ and $\left(U,\|\cdot\|_U\right)$  and a normed vector space of inputs $U_c\subset\left\{u : \mathbb{R}_{\geq 0}\rightarrow U\right\}$. We assume that the following two axioms hold  true:
    \begin{enumerate}[(i)]
      \item for all $u\in U_c$ and all $\tau\geq 0$ the time shift $u(\cdot+\tau)$ belongs to $U_c$ with $\|u\|_{U_c}\geq \|u(\cdot+\tau)\|_{U_c}$;
      \item for all $u_1,u_2\in U_c$ and for all $t>0$ the concatenation of $u_1$ and $u_2$ at time $t$, defined by
         \begin{align*}
u_1\underset{t}{\diamond} u_2(\tau) = \left\{
\begin{array}{ll}
u_1(\tau), & \text{if} \tau \in [0,t] \\
u_2(\tau - t), & \text{otherwise}
\end{array}
\right.
\end{align*}
         belongs to $U_c$.
\end{enumerate}
Consider a transition map  $\phi: {D_{\phi}} \rightarrow X$  with $D_{\phi} \subseteq \mathbb{R}_{\geq 0}\times X\times U_c$. The triple $\Sigma$ is called a control system, if it verifies the following properties:
     \begin{enumerate}[(i)]
       \item identity property: for every $(x,u)\in X\times U_c$, it holds that $\phi(0,x,u)=x$;
       \item causality: for every $(t,x,u)\in \mathbb{R}_{\geq 0}\times X \times U_c$ and  $\widetilde{u}\in U_c$ satisfying $\widetilde{u}(\tau)=u(\tau)$ for $\tau\in[0,t]$, it holds that $\phi\left(t,x,\widetilde{u}\right)=\phi(t,x,u)$;
       \item continuity: for every $(x,u)\in X\times U_c$, the mapping $t\mapsto \phi(t,x,u)$ is continuous;
       \item cocycle property: for every $t,h\in \mathbb{R}_{\geq 0}$ and every $(x,u)\in X\times U_c$, it holds that
           \begin{align*}
           \phi\left(h,\phi(t,x,u),u(t+\cdot)\right)=\phi\left(t+h,x,u\right).
           \end{align*}
     \end{enumerate}
    \end{definition}

    The following definition is concerned with the forward complete control systems considered in this paper.

\begin{definition}{\cite{Mironchenko:2021}}
The control system $\Sigma=\left(X,U_c,\phi\right)$  is said to be forward-complete if for any $(t,x,u)\in \mathbb{R}_{\geq 0}\times X\times U_c$, the value $\phi(t,x,u)\in X$ is well-defined.
\end{definition}

The following definition is concerned with the generalized class-$\mathcal{KL}$  function ($\mathcal{GKL}$-function) used in this paper. Note that different  from the definition adopted in \cite{Lopez-Ramirez:2018},  the $\mathcal{GKL}$-function is defined in the same way as in \cite{Lopez-Ramirez:2020}.

\begin{definition}\cite{Lopez-Ramirez:2020}\label{q}
A continuous mapping $\beta$ : $\mathbb{R}_{\geq 0}\times\mathbb{R}_{\geq 0}\to\mathbb{R}_{\geq0} $ is called a  $\mathcal{GKL}$-function, if it satisfies the following conditions:
              \begin{enumerate}[(i)]
          \item the mapping $s\mapsto \beta(s,0)$ is a $\mathcal{K}$-function;
             \item for each fixed $s\in\mathbb{R}_{\geq0}$ the mapping $t\mapsto\beta(s,t)$ is continuous, decreases to zero and there exists a nonnegative and continuous function  $T(s) $ such that $\beta(s,t)=0$ for all $t\geq T(s)$.
             \end{enumerate}
\end{definition}

Now,  in accordance with the notion of FTISS  defined in \cite[Definition~4]{Lopez-Ramirez:2020} for finite-dimensional systems, we provide the definition of  FTISS for  the   system $\Sigma$.
\begin{definition}
  The control system $\Sigma=\left(X,U_c,\phi\right)$ is said to be {finite-time input-to-state stable (FTISS)}, if there exist functions $\beta\in\mathcal{GKL}$ and $\gamma\in\mathcal{K}$ such that
                            \begin{align*}
     \|\phi(t,x,u)\|_X\leq\beta\left(\|x\|_X,t\right)+\gamma\left(\|u\|_{U_c}\right),\forall (t,x,u)\in\mathbb{R}_{\geq 0}\times X\times U_c.
                           \end{align*}
\end{definition}
%\begin{remark}
%\end{remark}
\subsection{The FTISS Lyapunov theorem for infinite-dimensional systems}\label{3.2}

 For a real-valued function $\Psi : \mathbb{R}_{\geq0} \rightarrow \mathbb{R}$, the right-hand upper Dini derivative at $t\in \mathbb{R}_{\geq0}$ is given by
                                                 \begin{align*}
                D^+\Psi(t)=\mathop{\lim\sup}\limits_{s\to 0^+}\frac{\Psi(t+s)-\Psi(t)}{s}.
                                     \end{align*}
Let $x\in X$ and $V$ be a real-valued function defined in a neighborhood of $x$, the Lie derivative of $V$ at $x$ corresponding to the input $u\in U_c$ along the trajectory of the system $\Sigma$ is defined by
	                              \begin{align*}
        \dot V_u(x)=&\left.D^+V\left(\phi(\cdot,x,u)\right)\right|_{t=0}
                     =\mathop{\lim\sup}\limits_{t\to 0^+}\frac{V\left(\phi(t,x,u)\right)-V(x)}{t}.
                          \end{align*}
If it is clear from the context what the input   for computing the Lie derivative  $\dot V_u(x)$ is, then we simply write $\dot V(x)$.

We define the FTISS Lyapunov functional for the control system $\Sigma$.

\begin{definition}\label{n}
A continuous function $V:X\to \mathbb{R}_{\geq0}$ is called an FTISS Lyapunov functional (FTISS-LF) for the system $\Sigma$, if there exist constants $M\in\mathbb{R}_{>0}$ and $ {\sigma_0}\in(0,1)$ and functions $\alpha_1$, $\alpha_2\in \mathcal{K_\infty}$, and $\chi\in\mathcal{K}$ such that
     \begin{align}\label{i}
                \alpha_1\left(\| x \|_X\right)\leq V(x) \leq \alpha_2\left(\| x \|_X\right), \forall x\in X,
     \end{align}
and for any $u\in U_c$, the Lie derivative of $V$ at $x$ with respect to (w.r.t.) the input $u\in U_c$ along the trajectory satisfies
         \begin{align}\label{ii}
          \| x \| _X\geq \chi\left(\| u \|_{U_c}\right) \Rightarrow \dot V_u(x)\leq -M V^{{\sigma_0}}(x).
      \end{align}
% where $M$ is a positive number, $\alpha \in [0,1)$.
  \end{definition}

  %\begin{assumption}
   %Throughout this paper, we always propose the following assumptions:
%\begin{enumerate}
%\item[\textbf{(H1)}]   $U_c:=BC\left(\mathbb{R}_{\geq 0};U\right)$;
%\item[\textbf{(H2)}] $\phi=0$ is the unique equilibrium point of the control  system $\Sigma$;
%\item[\textbf{(H3)}]    The system $\Sigma=\left(X,U_c,\phi\right)$ is forward-complete.
%\end{enumerate}
%\end{assumption}

Throughout this paper, we always impose  the following assumptions:
\begin{enumerate}
\item[\textbf{(H1)}]   $U_c:=BC\left(\mathbb{R}_{\geq 0};U\right)$;
\item[\textbf{(H2)}] $\phi=0$ is the unique equilibrium point of the control  system $\Sigma$;
\item[\textbf{(H3)}]    The system $\Sigma=\left(X,U_c,\phi\right)$ is forward-complete.
\end{enumerate}

The following Lyapunov theorem is the first main result of this paper.
\begin{theorem}[FTISS Lyapunov theorem]\label{th}
If the control system $\Sigma=\left(X,U_c,\phi\right)$ admits an FTISS-LF, then it is FTISS.
\end{theorem}

\begin{proof}
Let $V(x)$ be an FTISS-LF of the control system $\Sigma$ with $\alpha_1$, $\alpha_2$, $\chi$, $M$, and $ {\sigma_0}$ being the same as in Definition~\ref{n}. Take an arbitrary control $u\in U_c$ and consider the set
                   \begin{align*}
                   \Omega:= \left\{x\in X\left|\right.~V(x)\leq \alpha_2\circ\chi\left(\|u\|_{U_c}\right) \right\}.
                 \end{align*}
  We claim that the set $\Omega$ is invariant, namely, as long as $  x_0\in \Omega $,   there must be   \begin{align*}
  x(t):=\phi(t,x_0,u)\in \Omega, \forall t\in \mathbb{R}_{\geq 0}.
   \end{align*}

  First, note that $0\in \Omega$.

    If $u\equiv0$, it follows from $\alpha_1,\alpha_2\in \mathcal{K_\infty}$, $\chi\in \mathcal{K}$, and~\eqref{i} that
         \begin{align*}
                 \Omega =&\left\{x\in X\left|\right.~V(x)\leq \alpha_2\circ\chi(0)\right\}
                         =\left\{x\in X\left|\right.~V(x)\leq 0\right\}
                         \subset\left\{x\in X\left|\right.~\alpha_1\left(\|x\|_X\right)\leq 0\right\}
                          =\{0\}.
              \end{align*}
     Therefore, $\Omega=\{0\}$. Since $x=0$ is an equilibrium point,   $\Omega$ is   invariant.

    If $u\not\equiv 0$, suppose that $\Omega$ is not invariant, then,   due to continuity of  $x$, there exists $t*\in \mathbb{R}_{\geq 0}$ such that
                                   \begin{align*}
                                           V(x(t*))=\alpha_2\circ\chi\left(\|u\|_{U_c}\right),
                                 \end{align*}
which, along with \eqref{i}, leads to
                                          \begin{align*}
                                          \| x(t*)\|_X\geq \chi\left(\|u\|_{U_c}\right).
                                 \end{align*}
Denote by  $\hat{u}$   the input to the system after $t*$, i.e., $\hat{u}(\tau)= u(\tau + t*)$ for all $\tau\in \mathbb{R}_{>0}$. It follows from \eqref{ii} that
                            \begin{align*}
                                \dot V_{\hat{u}}(x)\leq & -M V^{{\sigma_0}}(x)
                                           \leq  -M\left(\alpha_1\left(\|x\|_X\right)\right)^{{\sigma_0}}
                                          <  0, \forall t>t*.
                          \end{align*}
Therefore, the trajectory cannot escape from the set $\Omega$. This is a contradiction. We conclude that $\Omega$  is  invariant.

Now we consider $x_0\notin \Omega$  and let
\begin{align*}
\tilde{t}:=\inf\left\{t\in \mathbb{R}_{>0}|~V(x(t))\leq \alpha_2\circ\chi\left(\|u\|_{U_c}\right)\right\}.    \end{align*}

In view of \eqref{ii},  we have
                   \begin{align}\label{1}
             \dot V_u(x)\leq -M V^{{\sigma_0}}(x), \forall t\in\left(0,\tilde{t}\right).
          \end{align}
%where the constant $\alpha\in [0,1)$, $M\in\mathbb{R}_{>0}$.

It is clear that $V(x)\equiv0$ is the solution to \eqref{1}. However,   $V(x)\equiv0$ implies that $x(t)\equiv 0$ for all $t\in \left(0,\tilde{t}\right)$. By virtue of the continuity of $x$, we have $x_0=0\in \Omega$, which  leads to a contradiction. Then, we get  $V(x) \not\equiv 0$.
              %\begin{align*}
              % (V(x))^{-\alpha}\dot V_u(x)\leq -M.
              % \end{align*}

 Let $T_0:=\frac{{V^{1-{\sigma_0}}(x_0)}}{M(1-{\sigma_0})}$. If $T_0\in \left(0,\tilde{t}\right)$, we deduce from \eqref{1} that
   %\begin{enumerate}[(i)]
    % \item
   \begin{align*}
         V(x(t))\leq&\left(V^{1-{\sigma_0}}(x_0)-M(1-\sigma_0)t\right)^{\frac{1}{1-{\sigma_0}}},  \forall t\in\left(0,T_0\right),\\
         V(x(t))=&0,\forall t\in \left[T_0,\tilde{t}\right).
         \end{align*}

    % \item
    If $T_0\geq \tilde{t}$, we have
     \begin{align*}
         V(x(t))\leq \left(V^{1-{\sigma_0}}(x_0)-M(1-\sigma_0)t\right)^{\frac{1}{1-{\sigma_0}}},  \forall t\in\left(0, \tilde{t}\right).
         \end{align*}

   %\end{enumerate}
%Since $V(x)$ is a nonnegative function, when $t\geq T$, we have $V(x)=0$.

Define $T(s):=\frac{s^{1-\sigma_0}}{M(1-\sigma_0)}$ for $s\in\mathbb{R}_{\geq 0}$. Define the following $\mathcal{GKL}$-function:
                     \begin{align*}
                      \beta_1(s,t):=
                        \left\{\!
\begin{array}{ll}
\left(s^{1-{\sigma_0}} -M(1-{\sigma_0})t\right)^{\frac{1}{1-{\sigma_0}}}\!, & \!s\in \mathbb{R}_{\geq 0},t\in\left[0,T(s)\right), \\
\!0,&\! s\in \mathbb{R}_{\geq 0},t\in \left[T(s),+\infty\right).
\end{array}
\right.
                   \end{align*}

Then, for  $T_0\in \left(0,\tilde{t}\right)$ or $T_0\geq \tilde{t}$, we always have
  \begin{align*}
      V(x(t))\leq \beta_1\left(V(x_0),t\right), \forall t\in \left(0,\tilde{t}\right),
      \end{align*}
 which  implies that
                                        \begin{align}\label{a}
                       \|x\|_X\leq \beta\left(\|x_0\|_X,t\right),\forall t\in\left(0, \tilde{t}\right)
                  \end{align}
with
     \begin{align*}
  \beta(s,t):=\alpha_1^{-1}\circ \beta_1\left(\alpha_2(s),t\right),\forall s,t \in \mathbb{R}_{\geq0}.
  \end{align*}It is clear that  $\beta$ is  a $\mathcal{GKL}$-function.

By the definition of~$\tilde{t}$, we have
  $V\left(x\left(\tilde{t}\right)\right)= \alpha_2\circ\chi\left(\|u\|_{U_c}\right)$. Therefore, $x\left(\tilde{t}\right)\in\Omega$.

  Note that $\Omega$ is  invariant, we  deduce that
      \begin{align}\label{b}
    \|x\|_X\leq \rho\left(\|u\|_{U_c}\right), \forall t\in \left[\tilde{t},+\infty\right),
     \end{align}
 where $\rho:={\alpha_1}^{-1}\circ {\alpha_2}\circ\chi\in\mathcal{K}$.

By \eqref{a} and \eqref{b}, we have
              \begin{align*}
                  \|x\|_X\leq \beta\left(\|x_0\|_X,t\right)+\rho\left(\|u\|_{U_c}\right), \forall t\in \mathbb{R}_{\geq 0}.
               \end{align*}
We conclude that the system $\Sigma$ is   FTISS.
\end{proof}
\subsection{Constructing FTISS-LFs for a class of infinite-dimensional systems}\label{4}
In this section, we provide a sufficient condition for ensuring the existence of an FTISS-LF  for a class of infinite-dimensional systems under the framework of {compact semigroup theory} and Hilbert spaces. More precisely, letting $X$ be a Hilbert space with  scalar product $\langle \cdot,\cdot \rangle_X$ and  norm $\| \cdot \|_X :=\sqrt{\langle \cdot,\cdot \rangle_X}$  and $U\subset X$ be a  normed linear space, we consider the following system
           \begin{subequations}\label{sun}
       %\left\{
             \begin{align}
    &\dot{x}=Ax+\Xi(t,x,u), t\in\mathbb{R}_{> 0},\\
     &x(0)=x_0,
       \end{align}
              \end{subequations}
where $x\in X$ is the state,   $A:D(A)\rightarrow X$ is a linear operator, $\Xi :\mathbb{R}_{\geq 0}\times X\times U\rightarrow X$ is a nonlinear functional, $u\in U_c =C\left(\mathbb{R}_{\geq 0};U\right)$, and  $x_0\in X$ denotes the initial datum.
 Moreover, we always impose the following conditions:
\begin{enumerate}
\item[\textbf{(H4)}]    the operator $A$ is the infinitesimal generator of   a compact  $C_0$-semigroup    $S(t)$ on $X$ for $t\in \mathbb{R}_{\geq 0}$;
\item[\textbf{(H5)}]  $\Xi:\mathbb{R}_{\geq 0}\times X \times U\rightarrow X$ is continuous.
\end{enumerate}
%\end{assumption}

Recall that an  operator $P\in L(X)$ is said to be positive, if it is self-adjoint and satisfies
\begin{align*}\langle Px,x \rangle_X >0,\forall x\in X \setminus\{0\}.\end{align*}
%We call an operator-valued function $P\in L(X)$ positive, if $P$ is self-adjoint and positive.
An operator $P\in L(X)$ is called coercive, if   there exists    $\mu_1 \in \mathbb{R}_{>0}$ such that
\begin{align}
\langle Px,x \rangle_X \geq \mu_1\|x\|^2_{X}, \forall x\in X.\label{P-coervice}
\end{align}

The following theorem is the second main result of this paper. It indicates the well-posedness of system~\eqref{sun} and provides a sufficient condition for  ensuring the existence of an FTISS-LF and hence, the FTISS of  system~\eqref{sun}.
\begin{theorem}\label{th2}
Let the conditions \textbf{(H4)} and \textbf{(H5)} be fulfilled. Assume that there exist  a coercive and positive operator $P\in L(X)$, a function $\zeta \in\mathcal{K}$, and constants $b,c\in\mathbb{R}_{>0}$ and $\tau\in\left(1,2\right)$,  such that
                               \begin{align}\label{ro}
              \langle (PA+A^*P)x,x \rangle_X+2\langle \Xi(t,x,u),Px \rangle_X\leq -b\|x\|^{\tau}_X+c\|x\|_X\zeta\left(\|u\|_{U_c}\right),
                                       \end{align}
 holds true for all $t\in\mathbb{R}_{>0}$, $x\in D(A)\subset X$, and $u\in U_c$,
where $A^*$ denotes the adjoint operator of $A$. Then,   system  \eqref{sun} {admits a mild solution $x\in C\left(\mathbb{R}_{\geq0};X\right)$, which is  defined by}
                             \begin{align}
                              x(t)=S(t)x_0+\int_0^t S(t-s)\Xi(s,x(s),u(s)) \mathrm{d}s,\forall t\in\mathbb{R}_{\geq 0}.\label{mild-sol}
                              \end{align}
                              Moreover,
                              the functional $ V(x):=\langle Px,x \rangle_X $ is an FTISS-LF and hence, system \eqref{sun} is FTISS.
 \end{theorem}
\begin{proof}
 {Note that under  the assumptions~\textbf{(H4)} and \textbf{(H5)},   \cite[Corollary 2.3,  p.194]{Pazy:2012}   ensures  that, for every   $x_0\in X$  and  $u\in U_c$,  system \eqref{sun} admits a mild solution $x\in C\left(\mathbb{R}_{\geq0};X\right)$, which is  defined by \eqref{mild-sol}.}
                            % \begin{align*}
%                              {x(t)=S(t)x_0+\int_0^t S(t-s)\Xi(s,x(s),u(s)) \mathrm{d}s,\forall t\in\mathbb{R}_{\geq 0}.}
%                              \end{align*}

               Now, for the mild solution $x$, we prove that the functional $ V(x):=\langle Px,x \rangle_X $  is an FTISS-LF.

Since $P\in L(X)$ is coercive, there exists $\mu_1 \in\mathbb{R}_{>0}$ such that
                      \begin{equation}\label{ong}
       \mu_1\|x\|^2_X\leq V(x)\leq\mu_2\|x\|^2_X, \forall x\in X,
                  \end{equation}
where $\mu_2:=\|P\|$.

  By direct calculation, we have
              \begin{align*}
                    \dot{V}_u(x)=&\langle P\dot{x},x \rangle_X+\langle Px,\dot{x} \rangle_X\\
                                 =&\langle P(Ax+\Xi(t,x,u)),x \rangle_X+\langle Px, Ax+\Xi(t,x,u) \rangle_X\\
                                =&\langle PAx,x \rangle_X+\langle P\Xi(t,x,u),x \rangle_X +\langle Px,Ax \rangle_X+\langle Px,\Xi(t,x,u) \rangle_X\\
                                  =&\langle (PA+A^*P)x,x\rangle_X+\langle \Xi(t,x,u),P^*x \rangle_X+\langle Px,\Xi(t,x,u) \rangle_X\\
                                =&\langle (PA+A^*P)x, x\rangle_X+2\langle  \Xi(t,x,u),Px \rangle_X.
                        \end{align*}

By \eqref{ro}, we have
                         \begin{align}
     \dot{V}_u(x) \leq - b\|x\|^{\tau}_X+ c\|x\|_X\zeta\left(\|u\|_{U_c}\right).\label{Vdot}
                 %&\leq-\|x\|^2_X-2a\|P\|\|x\|^{1+r}_X+\|x\|^2_X\\
%                    &\quad+b^2\|P\|^2\zeta^2(\|u\|_{U_c})\\
%                  &=-2a\|P\|\|x\|^{1+r}_X+b^2\|P\|^2\zeta^2(\|u\|_{U_c}).
            \end{align}

Let $\epsilon_0\in (0,b)$ be an arbitrary constant. Define the $\mathcal{K}$-function
$\chi(s):= {\left(\frac{c}{\epsilon_0} \zeta (  s ) \right)^{\frac{1}{\tau-1}}} $ for any $s\in \mathbb{R}_{\geq 0}$. It follows that
\begin{align*}
\|x\|_X\geq \chi\left(\|u\|_{U_c}\right) \Rightarrow  c \|x\|_X\zeta\left(\|u\|_{U_c}\right) \leq {\epsilon_0} \|x\|^{\tau}_X,
\end{align*}
which, along with \eqref{Vdot}, yields
\begin{align}
   \|x\|_X \geq  \chi\left(\|u\|_{U_c}\right)  \Rightarrow    \dot{V}_u(x)  \leq    -(b-\epsilon_0) \|x\|^{\tau}_X.\label{eqn:17}
\end{align}

Note that \eqref{ong} gives
\begin{align}
             \|x\|_X\geq \mu_2^{-\frac{1}{2}}V^{\frac{1}{2}}(x)   .\label{eqn:18}
 \end{align}

We infer from \eqref{eqn:17} and \eqref{eqn:18} that
 \begin{align*}
\|x\|_X\geq \chi\left(\|u\|_{U_c}\right) \Rightarrow   \dot{V}_u(x)  \leq    -MV^{{\sigma_0}}(x),
\end{align*}
where
 \begin{align*}
 \sigma_0:= \frac{\tau}{2}\in (0,1) \quad \text{and}\quad
  M:=  {(b-\epsilon_0)}  \mu_2^{-\frac{\tau}{2}}>0.
 \end{align*}
Therefore,  $ V(x)=\langle Px,x \rangle_X $ is an FTISS-LF for system~\eqref{sun}. Furthermore,  Theorem \ref{th} ensures that system \eqref{sun} is FTISS.
% \hfill $\blacksquare$
\end{proof}
 \begin{remark}\label{Rem-1} The disturbance-free system~\eqref{sun}, i.e., system~\eqref{sun} with $u\equiv 0$, is finite-time stable in the finite time  $T :=\frac{{V^{1-\sigma_0}(x_0)}}{M(1-\sigma_0)}$.
Furthermore, by virtue of the arbitrariness of $\epsilon_0\in (0,b)$, the settling time, denoted by $T_*$,  satisfies
 \begin{align*}
T_*\leq &\lim_{\epsilon_0\to 0^+}T
    =  \frac{2\mu^{\frac{\tau}{2}}_2}{ (2-\tau)b}V^{1-\frac{\tau}{2}}(x_0),
\end{align*}
{which may depend on the initial data and hence, it cannot be prescribed in advance.}
 \end{remark}
 \begin{remark}\label{rem.2}For finite-dimensional systems containing sublinear terms, it is a relatively easy task to verify  the structural condition  \eqref{ro} and establish the FTISS of the systems; see, e.g., \cite{Lopez-Ramirez:2020}. However, for infinite-dimensional systems described by specific PDEs, {as will be shown in Section~\ref{g}}, even if the PDEs contain  sublinear terms, verifying the structural condition   \eqref{ro}  remains challenging and needs  more tools.
 \end{remark}
 \section{FTISS for a class of  parabolic PDEs}\label{g}
In this section,   we show how to  verify the FTISS property for a class of parabolic PDEs with distributed in-domain disturbances by using the FTISS Lyapunov theorem, i.e., Theorem~\ref{th2}. In addition, we conduct numerical simulations to illustrate the obtained theoretical result. More precisely,  we consider the following nonlinear parabolic equation with in-domain disturbances and homogeneous mixed boundary conditions:
\begin{subequations}\label{sys}
         \begin{align}
%\left\{
 %     \begin{aligned}
     w_t(y,t)=&  w_{yy}(y,t)-k\left|w(y,t)\right|^{r-1}w(y,t)
            +f(y,t), (y,t)\in(0,1)\times \mathbb{R}_{>0},\\
    w(0,t)=&0,t\in\mathbb{R}_{>0},\\
    w_y(1,t)=&0,t\in\mathbb{R}_{>0},\\
     w(y,0)=&w_0(y),y\in(0,1),
%\end{aligned}\right.
          \end{align}\end{subequations}
where {$k\in\mathbb{R}_{>0}$},   $r\in(0,1)$, the function $f\in C\left(\mathbb{R}_{\geq 0};L^2(0,1)\right)$ represents distributed in-domain disturbance, and the function $w_0\in L^2(0,1)$ represents the initial datum.

%\subsection{Analysis of well-posedness}

Let $X:=L^2(0,1)$,  $U_{c}:=C\left(\mathbb{R}_{\geq 0};L^2(0,1)\right)$.  We express system~\eqref{sys} under the abstract form:
\begin{subequations}\label{sys'}
       %\left\{
\begin{align}
    &\dot{w}=Aw+\Xi(t,w,f), t\in\mathbb{R}_{>0},\\
     &w(0)=w_0,
\end{align}
       %\right.
\end{subequations}
where the  linear operator $A$ is defined by
                           \begin{align}
                            Aw:=  w_{yy}(y), \forall w\in D(A):=\left\{w\in H^2(0,1)\left|\right.~w(0)=0, w_y(1)=0\right\},\label{operator-A}
                           \end{align}
%with $D(A):=\{w\in H^2(0,1)\left|\right.w(0)=0, w_y(1)=0\}$,
 and the nonlinear functional $\Xi$ is defined by
\begin{align*}
 \Xi(t,w,f):=-k\left|w\right|^{r-1}w +f , \forall t\in\mathbb{R}_{\geq0}, w\in X,f\in U_c.
 \end{align*}
 {It is well known that   the   operator $A$ is the infinitesimal generator of   a  $C_0$-semigroup    $S(t)$ on $X$ for $t\in \mathbb{R}_{\geq 0}$.}

 %For any given $f\in U_c$, it is clear that
%\begin{align}
%\|\Xi(t,w,f)\|_{X}\leq k\| \left|w\right|^r\|_X+\| f\|_{X}\leq k\| w\|_X+ kC(r)+\| f\|_{U_c},\forall t\in\mathbb{R}_{\geq0}, w\in X,\label{struct-cond.}
%\end{align}
%where in the last inequality we used the Young's inequality (see \cite[Appendix B.2, p.706]{Evans:2010}), and $C(r)$ is {a} positive constant depending only on $r$.

%The following theorem, which states the well-posedness and the FTISS of system~\eqref{sys'}, is the third main result obtained in this paper.  Its proof will be divided into two parts and given in Section~\ref{Well-posedness} and Section~\ref{Sec: proof of Prop.}, respectively.
%\begin{theorem}\label{prop}  %Assume that  $f\in C(\mathbb{R}_{\geq 0};L^2(0,1))$ and $ w_0\in L^2(0,1)$.
%System~\eqref{sys'} admits a  mild solution $w\in C\left(\mathbb{R}_{\geq0};X\right)$ and is FTISS in the spatial $L^2$-norm w.r.t. the in-domain disturbance $f$.
%\end{theorem}

%\subsection{FTISS assessment}

%In this section, we prove Proposition~\ref{prop} by applying the FTISS Lyapunov theorem.

\subsection{Well-posedness analysis}\label{Well-posedness}

We   present the following proposition, which indicates the   well-posedness of system~\eqref{sys'}, or, equivalently, system~\eqref{sys}.

\begin{proposition}\label{Prop.-well-posedness}System~\eqref{sys'} admits a  mild solution $w\in C\left(\mathbb{R}_{\geq0};X\right)$.
\end{proposition}
%Before proving Proposition~\ref{prop}, let's first verify the existence of solutions for System~\eqref{sys}.
 Note that the functional $ \Xi(t,w,f)$ has a sublinear term $\left|w\right|^{r-1}w$ for $r\in(0,1)$ and hence, is not Lipschitz continuous w.r.t. $w$. Thus, as indicated in \cite[p.191]{Pazy:2012}, the strong continuity of the semigroup $S(t)$
is not sufficient for ensuring the existence of a mild solution to system~\eqref{sys'}.  In this case, we need to verify a stronger property of the semigroup $S(t)$. More precisely, we prove the following lemma, which ensures the compactness and analyticity of   $S(t)$.
\begin{lemma}\label{le2} Let the operator $A$ be defined by \eqref{operator-A}. Then, the operator $A$
       is the infinitesimal generator of a compact analytic semigroup $S(t)$  on $X$ for $t\in\mathbb{R}_{\geq 0}$.

       \end{lemma}
\begin{proof} We   prove Lemma~\ref{le2}  in a similar way as in the proof of {\cite[Lemma~2.1,  pp. 234-235]{Pazy:2012}}. Letting $ g \in X$ and $\lambda=\rho e^{i \theta}$ with $\rho \in \mathbb{R}_{>0}$ and $ \theta\in \left(-\frac{\pi}{2},\frac{\pi}{2}\right)$, consider the boundary value problem:
\begin{subequations}\label{r11}
                                                \begin{align}
                                      \lambda^2 u-u''=&g,  \\
                                      u(0)=&0,\\
                                      u'(1)=&0.
                                     \end{align}
          \end{subequations}
%where $u\in X$.\\

%First, we solve the homogeneous equation $\lambda^2u-u''=0$. The characteristic equation for this differential equation is $\lambda^2=r^2$. Solving the characteristic equation, we find the roots $r=\pm\lambda$. Therefore, the general solution to the homogeneous equation is given by
%\begin{align}\label{r12}
%u_n(x)=A \cosh(\lambda x)+B\sinh (\lambda x),
%\end{align}
%where $A$ and $B$ are arbitrary constants.\\
%Combining the boundary conditions of system~\eqref{r11}, we have
%                                             \begin{align*}
%                                             A\cosh(0)+B\sinh(0)=&0,\\
%                                             \lambda B\cosh(\lambda)=&0.
%                                             \end{align*}
%From these equations, we get $A=0, B=0$, which implies that $u_n=0$.\\
%Next, we use the method of variation of parameters to solve the non-homogeneous equation $\lambda^2u-u''=g$. We assume the solution is $u(x)=u_n(x)+u_p(x)$,  where $u_p(x)$  is the particular solution to the non-homogeneous equation and is given by
     Let $G(x,y)$ be the Green's function that satisfies the following conditions
                                  % \begin{subequations}\label{r13}
         \begin{align*}
                               \lambda^2 G(x,y)-G_{xx}(x,y)=&\delta(x-y), \\
                               G(0,y)=&0, \\
                               G_x(1,y)=&0,
                               \end{align*}
                            %   \end{subequations}
where $\delta(\cdot)$  is the standard Dirac delta function.

By direct computations, we otain
\begin{align*}
              G(x,y)=
                        \left\{\begin{aligned}
\frac{\sinh(-\lambda +\lambda y+\lambda x)+\sinh(\lambda -\lambda y+\lambda x)}{2\lambda \cosh(\lambda)}, &&  x<y, \\
 \frac{\sinh(-\lambda+\lambda y+\lambda x)+\sinh(\lambda y+\lambda -\lambda x)}{2\lambda \cosh(\lambda)}, && x>y.
\end{aligned}\right.
                   \end{align*}
                   Then, the solution to      \eqref{r11} is given by
                                         \begin{align*}
                                         u (x)=&\int_0^1 G(x,y)g(y) \mathrm{d}y\notag\\
                                         %\end{align*}
%Then the solution of system~\eqref{r11} is
%                                      \begin{align*}
%                                 u(x)=&\int_0^x  \left(C(y)e^{\lambda x}+D(y)e^{-\lambda x}\right)g(y)\mathrm{d}y+\int_x^1\left(A(y)e^{\lambda x}+B(y)e^{-\lambda x}\right)g(y)\mathrm{d}y\notag\\
                                    =&\int_0^x \left(\frac{\sinh(-\lambda+\lambda y+\lambda x)+\sinh(\lambda y+\lambda -\lambda x)}{2\lambda \cosh(\lambda)}\right)g(y)\mathrm{d}y\notag\\
                                    &+\int_x^1 \left(\frac{\sinh(-\lambda +\lambda y+\lambda x)+\sinh(\lambda -\lambda y+\lambda x)}{2\lambda \cosh(\lambda)}\right)g(y)\mathrm{d}y\notag\\
                                    =& \int_0^1 \frac{\sinh(\lambda(y+x-1))}{2\lambda \cosh(\lambda)}g(y)\mathrm{d}y+\int_0^x \frac{\sinh(\lambda(y+1-x))}{2\lambda \cosh(\lambda)}g(y)\mathrm{d}y\notag\\
                                    &+\int_x^1 \frac{\sinh(\lambda(1-y+x))}{2\lambda \cosh(\lambda)}g(y)\mathrm{d}y,\forall x\in [0,1].
                                    \end{align*}
It follows that
      \begin{align}\label{r25}
\left|u(x)\right| \leq &\int_0^1 \left|\frac{\sinh(\lambda(y+x-1))}{2\lambda \cosh(\lambda)}\right|\left|g(y)\right|\mathrm{d}y+\int_0^x \left|\frac{\sinh(\lambda(y+1-x))}{2\lambda \cosh(\lambda)}\right|\left|g(y)\right|\mathrm{d}y\notag\\
&+\int_x^1 \left|\frac{\sinh(\lambda(1-y+x))}{2\lambda \cosh(\lambda)}\right|\left|g(y)\right|\mathrm{d}y\notag\\
\leq & \frac{\|g\|_X}{2|\lambda||\cosh(\lambda)|} \left(\int_0^1 \left|\sinh(\lambda(y+x-1))\right|^2 \mathrm{d}y\right)^{\frac{1}{2}}+\frac{\|g\|_X}{2|\lambda||\cosh(\lambda)|}\left( \int_0^x\left|\sinh(\lambda(y+1-x))\right|^2 \mathrm{d}y\right)^{\frac{1}{2}}\notag\\
&+\frac{\|g\|_X}{2|\lambda||\cosh(\lambda)|}\left( \int_x^1 \left|\sinh(\lambda(1-y+x))\right|^2 \mathrm{d}y\right)^{\frac{1}{2}}
 ,\forall x\in [0,1].
     \end{align}

We need to estimate each term on the right-hand side of \eqref{r25}.

First, note that, for a complex number $z:=a+b\text{i}$ with $a,b\in\mathbb{R}$ and $\sqrt{\text{i}}=-1$, we have
                                         \begin{align}\label{z1}
                                         \left|\sinh(z)\right|=&\left|\frac{e^{z}-e^{-z}}{2}\right|  \notag\\
                                         =&\left|\frac{e^{a}(\cos(b)+\text{i}\sin(b))-e^{-a}(\cos(-b)+\text{i}\sin(-b))}{2}\right| \notag\\
                                         =& \sqrt{\left(\frac{e^{a}\cos(b)-e^{-a}\cos(-b)}{2}\right)^2+\left(\frac{e^{a}\sin(b)-e^{-a}\sin(-b)}{2}\right)^2} \notag\\
                                        \leq& \sqrt{\left(\frac{e^{a}-e^{-a}}{2}\right)^2+\left(\frac{e^{a}+e^{-a}}{2}\right)^2} \notag\\
                                       =& \sqrt{\frac{e^{2a}+e^{-2a}-2}{4}+\frac{e^{2a}+e^{-2a}+2}{4}} \notag\\
                                         =& \sqrt{\cosh(2a)}
                                         \end{align}
and
                             \begin{align}\label{z2}
                                         \left|\cosh(z)\right|=&\left|\frac{e^{z}+e^{-z}}{2}\right|  \notag\\
                                         =&\left|\frac{e^{a}(\cos(b)+\text{i}\sin(b))+e^{-a}(\cos(-b)+\text{i}\sin(-b))}{2}\right| \notag\\
                                         =& \sqrt{\left(\frac{e^{a}\cos(b)+e^{-a}\cos(-b)}{2}\right)^2+\left(\frac{e^{a}\sin(b)+e^{-a}\sin(-b)}{2}\right)^2} \notag\\
                                         \geq& \sqrt{\left(\frac{e^{a}\cos(b)+e^{-a}\cos(-b)}{2}\right)^2} \notag\\
                                         =& \cosh(a)\cos(b).
                                         \end{align}

   % for the first expression on the right-hand side of inequality \eqref{r25}, we have the following result through the H\"{o}lder's inequality (see \cite[p.706, Appendix B]{Evans:2010})
%                                  \begin{align}\label{e1}
%                   \int_0^1 \left|\frac{\sinh(\lambda(y+x-1))}{2\lambda \cosh(\lambda)}g(y)\right|\mathrm{d}y \leq  \frac{\|g\|_X}{2|\lambda||\cosh(\lambda)|} \left(\int_0^1 \left|\sinh(\lambda(y+x-1))\right|^2 \mathrm{d}y\right)^{\frac{1}{2}}.
%                   %%% =&\frac{\|g\|_X}{2|\lambda||\frac{e^{\lambda}+e^{-\lambda}}{2}|} \left(\int_0^1 \left|\frac{e^{\lambda(y+x-1)}-e^{-\lambda(y+x-1)}}{2}\right|^2 \mathrm{d}y\right)^{\frac{1}{2}}
%                    \end{align}

    Denote $\mu:=\text{Re}\lambda =\rho \cos(\theta)>0$. For the first term on the right-hand side of \eqref{r25}, setting $z:=\lambda (y+x-1)$ in \eqref{z1} and \eqref{z2},   we deduce that
                  \begin{align}\label{r26}
                      \int_0^1 \left|\frac{\sinh(\lambda(y+x-1))}{2\lambda \cosh(\lambda)}\right|\left|g(y)\right|\mathrm{d}y
                    %%% \leq & \frac{\|g\|_X}{2|\lambda|\left|\frac{e^{\mu}\left(\cos(\theta)+i\sin(\theta)\right)+e^{-\sigma}\left(\cos(-\theta)+i\sin(-\theta)\right)}{2}\right|} \notag\\
                      % &\left(\int_0^1 \left|\frac{e^{\sigma(y+x-1)}\left(\cos(\theta(y+x-1))+i\sin(\theta(y+x-1))\right)-e^{-\sigma(y+x-1)}\left(\cos(-\theta(y+x-1))+i\sin(-\theta(y+x-1))\right)}{2}\right|^2 \mathrm{d}y\right)^{\frac{1}{2}}\notag\\
           %%% =&\frac{\|g\|_X}{2|\lambda|\sqrt{\left(\frac{e^{\mu}\cos(\theta)+e^{-\sigma}\cos(-\theta)}{2}\right)^2+\left(\frac{e^{\mu}\sin(\theta)+e^{-\sigma}\sin(-\theta)}{2}\right)^2}} \notag\\
                      %%% &\left(\int_0^1  \left(\frac{e^{\sigma(y+x-1)}\cos(\theta(y+x-1))-e^{-\sigma(y+x-1)}\cos(-\theta(y+x-1))}{2}\right)^2\right.\notag\\
                       %%% &\left.+\left(\frac{e^{\sigma(y+x-1)}\sin(\theta(y+x-1))-e^{-\sigma(y+x-1)}\sin(-\theta(y+x-1))}{2}\right)^2 \mathrm{d}y\right)^{\frac{1}{2}}\notag\\
                    %%\leq & \frac{\|g\|_X}{2|\lambda|\sqrt{\left(\frac{e^{\mu}\cos(\theta)+e^{-\sigma}\cos(-\theta)}{2}\right)^2}}\left(\int_0^1 \left(\frac{e^{\sigma(y+x-1)}-e^{-\sigma(y+x-1)}}{2}\right)^2+\left(\frac{e^{\sigma(y+x-1)}-e^{-\sigma(y+x-1)}}{2}\right)^2 \mathrm{d}y\right)^{\frac{1}{2}}\notag\\
                    %%%=& \frac{\|g\|_X}{2|\lambda|\sqrt{\left(\frac{e^{\mu}+e^{-\sigma}}{2}\right)^2}\cos(\theta)}\left(\int_0^1 2\left(\frac{e^{\sigma(y+x-1)}-e^{-\sigma(y+x-1)}}{2}\right)^2\mathrm{d}y\right)^{\frac{1}{2}}\notag\\
                    \leq &\frac{\|g\|_X}{2|\lambda|\cosh(\mu)\cos(\theta)}\left(\int_0^1 \cosh(2\mu(y+x-1)) \mathrm{d}y\right)^{\frac{1}{2}}\notag\\
                    %%%=&\frac{\|g\|_X}{2|\lambda|\cosh(\mu)\cos(\theta)}\left(\int_0^1 \cosh 2(\mu(y+x-1))-1 \mathrm{d}y\right)^{\frac{1}{2}}\notag\\
                   = &\frac{\|g\|_X}{2|\lambda|\cosh(\mu)\cos(\theta)}\left( \frac{1}{2\mu}\left(\sinh(2\mu x)-\sinh(2\mu(x-1))\right) \right)^{\frac{1}{2}},\forall x\in[0,1].
                    \end{align}
Note that for all $x\in [0,1]$ we always have
 \begin{align}
\lim\limits_{\mu\rightarrow 0}\frac{\left( \frac{1}{2\mu}\left(\sinh(2\mu x)-\sinh(2\mu(x-1))\right) \right)^{\frac{1}{2}}}{\cosh(\mu)}
=&\lim\limits_{\mu\rightarrow0}\left(\frac{\sinh(2\mu x)-\sinh(2\mu(x-1))}{2\mu\cosh^2(\mu)}\right)^{\frac{1}{2}}\notag\\
=& \lim\limits_{\mu\rightarrow0} \left(\frac{e^{2\mu x}-e^{-2\mu x}-e^{2\mu(x-1)}+e^{-2\mu(x-1)}}{\mu(e^{2\mu}+e^{-2\mu}+2)}\right)^{\frac{1}{2}}\notag\\
=&\lim\limits_{\mu\rightarrow0}\left(\frac{2xe^{2\mu x}+2xe^{-2\mu x}-2(x-1)e^{2\mu(x-1)}}{(e^{2\mu}+e^{-2\mu}+2)+(\mu(2e^{\mu}-2e^{-2\mu}))}\right.\notag\\
        &-\left.\frac{2(x-1)e^{-2\mu(x-1)}}{(e^{2\mu}+e^{-2\mu}+2)+(\mu(2e^{\mu}-2e^{-2\mu}))}\right)^{\frac{1}{2}}\notag\\
=& 1 \label{limit-1}
\end{align}
and
                      \begin{align}\label{limit-2}
                         \lim\limits_{\mu\rightarrow  +\infty}  \frac{\left( \frac{1}{2\mu}\left(\sinh(2\mu x)-\sinh(2\mu(x-1))\right)\right)^{\frac{1}{2}}}{\cosh(\mu)}= &\lim\limits_{\mu\rightarrow  +\infty}\left(\frac{ \left(\sinh(2\mu x)-\sinh(2\mu(x-1)) \right)}{2\mu\cosh^2(\mu)}\right)^{\frac{1}{2}}\notag\\
                         =& \lim\limits_{\mu\rightarrow  +\infty} \left(\frac{e^{2\mu x}-e^{-2 \mu x}-e^{2\mu(x-1)}+e^{-2\mu(x-1)}}{\mu(e^{2\mu}+e^{-2\mu}+2)}\right)^{\frac{1}{2}}\notag\\
                         =& \lim\limits_{\mu\rightarrow  +\infty} \left(\frac{e^{2\mu x}+e^{-2\mu(x-1)}}{\mu(e^{2\mu}+2)}\right)^{\frac{1}{2}}\notag\\
                         =& 0.
                           \end{align}
%
%
%For any $\theta_0\in\left(\frac{\pi}{4},\frac{\pi}{2}\right)$, $\theta\in (0,2\theta_0)$, we get
%\begin{align*}
%\lim\limits_{\mu\rightarrow0} \frac{\left( \frac{1}{2\sigma}\left(\sinh(2\mu x)-\sinh(2\mu(x-1))\right)-1 \right)^{\frac{1}{2}}}{2\cosh(\mu)\cos(\theta)}= 0,
%\end{align*}
%and
%\begin{align*}
%\lim\limits_{\mu\rightarrow  +\infty} \frac{\left( \frac{1}{2\sigma}\left(\sinh(2\mu x)-\sinh(2\mu(x-1))\right)-1 \right)^{\frac{1}{2}}}{2\cosh(\mu)\cos(\theta)}= 0.
%\end{align*}
%By the properties of continuous functions, $\frac{\left( \frac{1}{2\sigma}\left(\sinh(2\mu x)-\sinh(2\mu(x-1))-1 \right) \right)^{\frac{1}{2}}}{2\cosh(\mu)\cos(\theta)}$  is bounded and we denote
%
We deduce that there exists $M_1\in \mathbb{R}_{>0}$   such that
 \begin{align*}%\label{m1}
                             \frac{\left( \frac{1}{2\mu}\left(\sinh(2\mu x)-\sinh(2\mu(x-1))\right) \right)^{\frac{1}{2}}}{2\cosh(\mu)}\leq M_1,\forall \mu\in \mathbb{R}_{>0},x\in [0,1],
                             \end{align*}
which, along with \eqref{r26}, ensures that
                        \begin{align}\label{rr1}
                        \int_0^1 \left|\frac{\sinh(\lambda(y+x-1))}{2\lambda \cosh(\lambda)}\right|\left|g(y)\right|\mathrm{d}y\leq \frac{M_1}{|\lambda|\cos(\theta)} \|g\|_X,\forall x\in[0,1].
                        \end{align}

                        For the second term on the right-hand side of \eqref{r25}, setting $z:=\lambda(y+1-x)$ in \eqref{z1} and \eqref{z2},   we deduce that
                          \begin{align}\label{r28}
                  \int_0^x \left|\frac{\sinh(\lambda(y+1-x))}{2\lambda \cosh(\lambda)}\right|\left|g(y)\right|\mathrm{d}y
         %%%=& \frac{\|g\|_X}{2|\lambda|\left|\frac{e^{\lambda}+e^{-\lambda}}{2}\right|}\left( \int_x^1 \left|\frac{e^{\lambda(y+1-x)}+e^{-\lambda(y+1-x)}}{2}\right|^2 \mathrm{d}y\right)^{\frac{1}{2}}\notag\\
            %%%  =& \frac{\|g\|_X}{2|\lambda|\sqrt{\left(\frac{e^{\mu}\cos(\theta)+e^{-\sigma}\cos(-\theta)}{2}\right)^2+\left(\frac{e^{\mu}\sin(\theta)+e^{-\sigma}\sin(-\theta)}{2}\right)^2}}\notag\\
            %%%  & \left( \int_0^x \left(\frac{e^{\mu(y+1-x)}\cos(\theta(y+1-x))-e^{-\mu(y+1-x)}\cos(-\theta(y+1-x))}{2}\right)^2 \right.\notag\\
            %%%  &\left. +\left(\frac{e^{\mu(y+1-x)}\sin(\theta(y+1-x))-e^{-\mu(y+1-x)}\sin(-\theta(y+1-x))}{2}\right)^2 \mathrm{d}y\right)^{\frac{1}{2}}               \notag\\
             %%%  \leq & \frac{\|g\|_X}{2|\lambda|\sqrt{\left(\frac{e^{\mu}\cos(\theta)+e^{-\sigma}\cos(-\theta)}{2}\right)^2}}\notag\\
              %%% & \left( \int_0^x \left(\frac{e^{\mu(y+1-x)}-e^{-\mu(y+1-x)}}{2}\right)^2  +\left(\frac{e^{\mu(y+1-x)}-e^{-\mu(y+1-x)}}{2}\right)^2 \mathrm{d}y\right)^{\frac{1}{2}}               \notag\\
               \leq & \frac{\|g\|_X}{2|\lambda|\cosh(\mu)\cos(\theta)}\left(\int_0^x \cosh(2\mu(y+1-x))\mathrm{d}y\right)^{\frac{1}{2}}\notag\\
              % =& \frac{\|g\|_X}{2|\lambda|\cosh(\mu)\cos(\theta)}\left( \int_0^x \left(\cosh(2\mu(y+1-x))-1\right) \mathrm{d}y\right)^{\frac{1}{2}}\notag\\
               =& \frac{\|g\|_X}{2|\lambda|\cosh(\mu)\cos(\theta)} \left( \frac{1}{2\mu} \left(\sinh(2\mu)- \sinh(2\mu (1-x))\right)\right)^{\frac{1}{2}},\forall x\in[0,1].
            \end{align}
Analogous  to the proof of \eqref{limit-1} and \eqref{limit-2}, we infer that
                  %\begin{align*}
%      \lim\limits_{\mu\rightarrow0}\frac{\left( \frac{1}{2\sigma} \sinh(2\mu)- \frac{1}{2\sigma}\sinh(2\mu x)-(1-x)\right)^{\frac{1}{2}}}{ \cosh(\mu)}
%      %=& \lim\limits_{\mu\rightarrow0}\left(\frac{\sinh(2\mu)-\sinh(2\mu x)-2\sigma (1-x)}{2\mu\cosh^2(\mu)}\right)^{\frac{1}{2}}\notag\\
%%      =& \lim\limits_{\mu\rightarrow0}\left(\frac{e^{2\mu}-e^{-2\mu}-e^{2\mu x}+e^{-2\mu x}-4\sigma (1-x)}{\sigma(e^{2\mu}+e^{-2\mu}+2)}\right)^{\frac{1}{2}}\notag\\
%%      =& \lim\limits_{\mu\rightarrow0} \left(\frac{2e^{2\mu}+2e^{-2\mu}  -2x e^{2\mu x}-2x e^{-2\mu x}-4(1-x)}{(e^{2\mu}+e^{-2\mu}+2)+\sigma(2e^{2\mu}-2e^{-2\mu}+2)}\right)^{\frac{1}{2}} \notag\\
%      =& 0
%      \end{align*}
%and
%                    \begin{align*}
%\lim\limits_{\mu\rightarrow  +\infty}\frac{\left( \frac{1}{2\sigma} \sinh(2\mu)- \frac{1}{2\sigma}\sinh(2\mu x)-(1-x)\right)^{\frac{1}{2}}}{ \cosh(\mu)}
% %=& \lim\limits_{\mu\rightarrow  +\infty}\left(\frac{\sinh(2\mu)-\sinh(2\mu x)-4\sigma (1-x)}{2\mu\cosh^2(\mu)}\right)^{\frac{1}{2}}\notag\\
%%      =& \lim\limits_{\mu\rightarrow  +\infty}\left(\frac{e^{2\mu}-e^{-2\mu}-e^{2\mu x}+e^{-2\mu x}-4\sigma (1-x)}{\sigma(e^{2\mu}+e^{-2\mu}+2)}\right)^{\frac{1}{2}}\notag\\
%%=&  \lim\limits_{\mu\rightarrow  +\infty}\left( \frac{e^{2\mu}-e^{2\mu x}-4\sigma (1-x)}{\sigma(e^{2\mu}+2)}\right)^{\frac{1}{2}}\\
%=&0
%                       \end{align*}
 \begin{align*}
              \lim\limits_{\mu\rightarrow0}\frac{\left(\frac{1}{2\mu} \left(\sinh(2\mu)-\sinh(2\mu(1-x))\right) \right)^{\frac{1}{2}}}{\cosh(\mu)}=& \lim\limits_{\mu\rightarrow0} \left(\frac{e^{2\mu}-e^{-2\mu}-e^{2\mu(1-x)}+e^{-2\mu(1-x)}}{\mu(e^{2\mu}+e^{-2\mu}+2)}\right)^{\frac{1}{2}} \notag\\
              =&  \lim\limits_{\mu\rightarrow0} \left(\frac{2 e^{2\mu}+2 e^{-2\mu}-2(1-x)e^{2\mu(1-x)}-2(1-x)e^{-2\mu(1-x)}}{(e^{2\mu}+e^{-2\mu}+2)+\mu(2e^{2\mu}-2e^{-2\mu})}\right)^{\frac{1}{2}} \notag\\
              =& \sqrt{x}
              \end{align*}
and
 \begin{align*}
              \lim\limits_{\mu\rightarrow  +\infty}\frac{\left(\frac{1}{2\mu} \left(\sinh(2\mu)-\sinh(2\mu(1-x))\right) \right)^{\frac{1}{2}}}{\cosh(\mu)}=& \lim\limits_{\mu\rightarrow  +\infty} \left(\frac{e^{2\mu}-e^{-2\mu}-e^{2\mu(1-x)}+e^{-2\mu(1-x)}}{\mu(e^{2\mu}+e^{-2\mu}+2)}\right)^{\frac{1}{2}} \notag\\
              =&  \lim\limits_{\mu\rightarrow  +\infty} \left(\frac{e^{2\mu}-e^{2\mu(1-x)}}{\mu(e^{2\mu}+2)}\right)^{\frac{1}{2}} \notag\\
              =& 0
\end{align*}
                       hold true for all $x\in [0,1]$.

%For any $\theta_0\in\left(\frac{\pi}{4},\frac{\pi}{2}\right)$, $\theta\in (0,2\theta_0)$, we have
%     \begin{align*}
%    \lim\limits_{\mu\rightarrow0} \frac{\left( \frac{1}{2\sigma} \sinh(2\mu(1+x))- \frac{1}{2\sigma}\sinh(2\mu)-x\right)^{\frac{1}{2}}}{2\cosh(\mu)\cos(\theta)}=0,
%    \end{align*}
%and
%          \begin{align*}
%           \lim\limits_{\mu\rightarrow  +\infty} \frac{\left( \frac{1}{2\sigma} \sinh(2\mu(1+x))- \frac{1}{2\sigma}\sinh(2\mu)-x\right)^{\frac{1}{2}}}{2\cosh(\mu)\cos(\theta)}= 0.
%           \end{align*}
%By the properties of continuous functions, $\frac{\left( \frac{1}{2\sigma} \sinh(2\mu(1+x))- \frac{1}{2\sigma}\sinh(2\mu)-x\right)^{\frac{1}{2}}}{2\cosh(\mu)\cos(\theta)}$  is bounded and we denote
Furthermore, we deduce   that there exists $M_2\in \mathbb{R}_{>0}$   such that
                 \begin{align*}%\label{m2}
              %   \frac{\left( \frac{1}{2\sigma} \sinh(2\mu(1+x))- \frac{1}{2\sigma}\sinh(2\mu)-x\right)^{\frac{1}{2}}}{2\cosh(\mu)\cos(\theta)}
                 \frac{\left(\frac{1}{2\mu} \left(\sinh(2\mu)-\sinh(2\mu(1-x))\right) \right)^{\frac{1}{2}}}{2\cosh(\mu)}
                 \leq M_2,\forall \mu\in \mathbb{R}_{>0},x\in [0,1],
               \end{align*}
which, along with \eqref{r28}, implies that
                   \begin{align}\label{e3}
           \int_0^x \left|\frac{\sinh(\lambda(y+1-x))}{2\lambda \cosh(\lambda)}\right|\left|g(y)\right|\mathrm{d}y
           \leq &  \frac{M_2}{|\lambda|\cos(\theta)} \|g\|_X,\forall x\in [0,1].
           \end{align}

 For the third term on the right-hand side of \eqref{r25}, analogous to the proof of \eqref{e3}, we deduce that $M_3\in \mathbb{R}_{>0}$   such that
                      \begin{align}\label{rr3}
                        \int_x^1 \left|\frac{\sinh(\lambda( 1-y+x))}{2\lambda \cosh(\lambda)}\right|\left|g(y)\right|\mathrm{d}y\leq \frac{M_3}{|\lambda|\cos(\theta)} \|g\|_X,\forall x\in [0,1].
                        \end{align}

Substituting \eqref{rr1}, \eqref{e3}, and \eqref{rr3} into \eqref{r25}, we obtain
                                       % \begin{align}\label{r29}
%                      |u(x)|\leq \frac{M_1\|g\|_X}{|\lambda|}+\frac{M_2\|g\|_X}{|\lambda|}+\frac{M_3\|g\|_X}{|\lambda|}.
%                      \end{align}
%By squaring both sides of inequality \eqref{r29} and then integrating the resulting inequality with respect to $x$ from $0$ to $1$, we can take the square root to obtain
                     \begin{align*}
                   \|u\|_X = \left( \int_0^1 \left|u(x)\right|^2 \mathrm{d}x\right)^{\frac{1}{2}}
                 %  \leq &  \left(\int_0^1 \left(\frac{M_1\|g\|_X}{|\lambda|}+\frac{M_2\|g\|_X}{|\lambda|}+\frac{M_3\|g\|_X}{|\lambda|}\right)^2 \mathrm{d}x\right)^{\frac{1}{2}}\notag\\
%                                                   = & \frac{M_1\|g\|_X}{|\lambda|}+\frac{M_2\|g\|_X}{|\lambda|}+\frac{M_3\|g\|_X}{|\lambda|}\notag\\
                                                   \leq  \frac{M_1+M_2+M_3}{|\lambda|\cos(\theta)}\|g\|_X.
                     \end{align*}

Fixing any $\theta_0\in\left(\frac{\pi}{4},\frac{\pi}{2}\right)$, we find that
                         \begin{align*}
                         \Sigma(\theta_0):=\{\lambda:|\arg\lambda|< 2\theta_0\}\subset \rho(A)
                        \end{align*}
and                     \begin{align*}
                     \|R(\lambda:A)\|\leq \frac{M}{|\lambda|} \quad \text{for} \quad \lambda\in\Sigma(\theta_0),
                    \end{align*}
                    where $M:=\frac{M_1+M_2+M_3}{\cos(\theta_0)}$.

%The operator $A$ be a densely defined in $X$, the relevant proof process is as follows.\\
%Since $H^2(0,1)$ is dense in $L^2(0,1)$, so for any $\psi \in L^2(0,1)$, $\epsilon >0$, there exist $\phi\in H^2(0,1)$, such that
 %                                                         \begin{align*}
 %                                                         \| \psi-\phi\|_{L^2(0,1)}<\epsilon.
 %                                                         \end{align*}
%we can construct a function $\hat{\phi}(y)=\phi(y)-\phi(0), \forall y\in (0,1)$, and $w(y)=\hat{\phi}(y)-y\hat{\phi}_y(1)$, which satisfies $w(0)=0$ and $w_y(1)=0$, so for any $w\in D(A), \epsilon >0$, and there exist $\phi\in H^2(0,1)$ such that
 %                                                             \begin{align*}
 %                                                             \| \phi-w\|_{L^2(0,1)}=\|\phi(0)+y\phi_y(1)\|_{L^2(0,1)}<\epsilon.
 %                                                             \end{align*}
%So for any $\psi \in L^2(0,1)$, $\epsilon >0$, there exist $w \in D(A)$ such that
%                                                 \begin{align*}
%                                                 \|\psi-w\|_{L^2(0,1)}<  \| \psi-\phi\|_{L^2(0,1)}+\| \phi-w\|_{L^2(0,1)} < 2\epsilon.
%                                                                                \end{align*}

 Note that $D(A)$ is dense in $X$. We infer from {\cite[Theorem 7.7,  p.~30]{Pazy:2012}} that $A$ is the infinitesimal generator of a $C_0$-semigroup $S(t)$ and satisfies
 \begin{align*}
  \|S(t)\|\leq C  ,\forall t\in\mathbb{R}_{\geq 0}
   \end{align*}
   with some positive constant $C$.
       Furthermore, we deduce from {\cite[Theorem 5.2, p.~61]{Pazy:2012}}    that  the semigroup $S(t)$ $\left(t\in\mathbb{R}_{\geq 0}\right)$ is  analytic.

       Finally,   the same process of   the proof of {\cite[Lemma~2.1,  pp.~234-235]{Pazy:2012}} indicates that the semigroup  $S(t)$ $\left(t\in\mathbb{R}_{\geq 0}\right)$ is compact. The proof is complete.
\end{proof}
\begin{proof}[Proof of Proposition~\ref{Prop.-well-posedness}.] For any given $f\in U_c$, it is clear that
\begin{align}
\|\Xi(t,w,f)\|_{X}\leq k\| \left|w\right|^r\|_X+\| f\|_{X}\leq k\| w\|_X+ kC(r)+\| f\|_{U_c},\forall t\in\mathbb{R}_{\geq0}, w\in X,\label{struct-cond.}
\end{align}
where in the last inequality we used the Young's inequality (see \cite[Appendix B.2, p.706]{Evans:2010}), and $C(r)$ is {a} positive constant depending only on $r$.

 In view of \eqref{struct-cond.} and   Lemma~\ref{le2},   we deduce from \cite[Corollary~2.3,  p.~194]{Pazy:2012} that system \eqref{sys'} admits a mild solution $w\in C(\mathbb{R}_{\geq0};X)$. The proof is complete.
\end{proof}
\subsection{FTISS assessment}\label{Sec: proof of Prop.}
In this section, we  show how to prove the  FTISS for system~\eqref{sys'}, or, equivalently, system~\eqref{sys}, when sublinear terms are involved. More precisely, we prove the following proposition, which is the third main result of this paper.
\begin{theorem}\label{prop}  %Assume that  $f\in C(\mathbb{R}_{\geq 0};L^2(0,1))$ and $ w_0\in L^2(0,1)$.
System~\eqref{sys'}, or, equivalently, system~\eqref{sys},  is FTISS in the spatial $L^2$-norm w.r.t. the in-domain disturbance $f$.
\end{theorem}

%\subsubsection{Interpolation inequalities}

 As indicated in Remark~\ref{rem.2},  verifying  the structural condition  \eqref{ro} for PDEs  is not an easy task. Therefore, before  proving Theorem~\ref{prop},   we prove for a function    $v\in W_{[0]}^{1,p}(0,1)$  some interpolation inequalities, which  can be used to establish the relationship between $ \|v\|^{1+s}_{L^{1+s}(0,1)}$ and $ -\|v\|^{2\tau}_{L^2(0,1)}$ with some $s,\tau\in (0,1)$ and hence, it plays   a crucial role in establishing the FTISS of parabolic PDEs with sublinear terms.
\begin{lemma}\label{lem} Let $v \in W_{[0]}^{1,p}(0,1)$.
  For any  $p,q\in(1,+\infty)$ and  $\delta\in \left(0,1\right)$ satisfying
                                  \begin{align}\label{x}
             \delta\left(\frac{1}{q}+1-\frac{1}{p}\right)=\frac{1}{q},
       \end{align}
 the following interpolation inequality holds true
                             \begin{flalign}\label{9}
                    \|v\|_{{L}^{\infty}(0,1)}\leq \frac{1}{{\delta}^{\delta}}\|v_y\|_{{L}^{p}(0,1)}^{\delta} \|v\|_{{L}^{q}(0,1)}^{1-\delta}.
                         \end{flalign}
\end{lemma}

\begin{proof}
Let  $p,q\in(1,+\infty)$ and $\delta\in \left(0,1\right)$ satisfy \eqref{x}.
          % \begin{align}
%             \delta\left(\frac{1}{q}+1-\frac{1}{p}\right)=\frac{1}{q}.
%            \end{align}
Let $\alpha:=\frac{1}{\delta}$. For any $s\in \mathbb{R}$, define $G(s):={\left|s\right|}^{\alpha-1}s$. For any $v \in W_{[0]}^{1,p}(0,1)$, we define the function $H(y):=G(v(y))$. It follows that
                           \begin{align*}
                      H_y(y)=G^{\prime}(v(y))v_y(y)=\alpha\left|v(y)\right|^{\alpha-1}v_y(y).
        \end{align*}
For any $y\in(0,1)$, we  deduce that
                  \begin{align}
               \left|H(y)\right|=&\left|\int_{0}^{y}H_z(z)\mathrm{d}z\right|
                     \leq\int_{0}^{y}\left|H_z(z)\right|\mathrm{d}z
                      =\int_{0}^{y}\left|\alpha\left|v(z)\right|^{\alpha-1}v_z(z)\right| \mathrm{d}z.\label{6}
                \end{align}
% From $\left(\ref{6}\right)$,

The equality~\eqref{x} ensures that
                                 \begin{align*}
                 \frac{1}{p}+\frac{1}{\frac{q}{\alpha-1}}=1
                                     \end{align*}
with $\frac{q}{\alpha-1}\in\left(1,+\infty\right)$.
 Then, for any $y\in(0,1)$, by using the H\"{o}lder's inequality (see {\cite[Appendix B.2, p.706]{Evans:2010})}, we obtain
                    \begin{align}
       \int_{0}^{y}\left|\alpha\left|v(z)\right|^{\alpha-1}v_z(z)\right|\mathrm{d}z\leq&\alpha{\left(\int_0^y{\left({\left|v(z)\right|}^{\alpha-1}\right)}^{\frac{q}{\alpha-1}}\mathrm{d}z\right)}^{\frac{\alpha-1}{q}}\left(\int_0^y{{\left|v_z(z)\right|}^{p}}\mathrm{d}z\right)^{\frac{1}{p}}\notag\\
        \leq&\alpha{\left(\int_0^1{\left({\left|v(z)\right|}^{\alpha-1}\right)}^{\frac{q}{\alpha-1}}\mathrm{d}z\right)}^{\frac{\alpha-1}{q}}\left(\int_0^1{{\left|v_z(z)\right|}^{p}}\mathrm{d}z\right)^{\frac{1}{p}}\notag\\
             =&\alpha\|v\|_{L^q(0,1)}^{\alpha-1}\|v_y\|_{L^p(0,1)}.\label{7}
                \end{align}

We infer from \eqref{6} and \eqref{7} that
            \begin{align}\label{11}
           \left|v(y)\right|^{\alpha}\leq\alpha\|v\|_{L^q(0,1)}^{\alpha-1}\|v_y\|_{L^p(0,1)},\forall y\in(0,1).
       \end{align}

Substituting $\alpha=\frac{1}{\delta}$ into \eqref{11}, it follows that
         \begin{align}\label{8}
       \left|v(y)\right|^{\frac{1}{\delta}}\leq \frac{1}{\delta}\|v\|_{L^q(0,1)}^{\frac{1-\delta}{\delta}}\|v_y\|_{L^p(0,1)},\forall y\in(0,1).
   \end{align}

Taking both sides of the inequality~\eqref{8} to the $\delta$-th power, we have
                \begin{align*}
       \left|v(y)\right|\leq\frac{1}{\delta^\delta}\|v\|_{L^q(0,1)}^{1-\delta}\|v_y\|_{L^p(0,1)}^{\delta},\forall y\in(0,1),
             \end{align*}
             which ensures that \eqref{9} holds true.
%Therefore,  we   get
%       \begin{align*}
%     \|v\|_{L^{\infty}(0,1)}\leq\frac{1}{\delta^\delta}\|v_y\|_{L^p(0,1)}^{\delta}\|v\|_{L^q(0,1)}^{1-\delta},\forall v \in W_{[0]}^{1,p}(0,1).
%        \end{align*}
%        \hfill $\blacksquare$
\end{proof}

\begin{corollary}\label{coro} Let $v \in W_{[0]}^{1,2}(0,1)$.
 For any $r\in(0,1)$, the following interpolation inequality holds true
                          \begin{align}
                        \|v\|_{L^2(0,1)}^{\frac{3+r}{2}} \leq&\frac{(3+r)\epsilon}{2}\|v_y\|_{L^2(0,1)}^2+ \frac{3+r}{8\epsilon}\|v\|_{L^{1+r}(0,1)}^{ 1+r},   \forall \epsilon\in\mathbb{R}_{>0}.\label{interpolation-estimate}
                      %\left(\int_0^1|v(y)|^{2} \mathrm{d}y\right)^{\frac{3+r}{4}}\leq& \frac{(3+r)\epsilon}{2}\int_0^1|v_y(y)|^2 \mathrm{d}y\\
%                             &+\frac{3+r}{8\epsilon}\int_0^1|v(y)|^{1+r} \mathrm{d}y,  \forall \epsilon\in\mathbb{R}_{>0}.
                          \end{align}
%where for any $\epsilon\in\mathbb{R}_{>0}$,  {$C_0=\frac{(3+r)\epsilon}{2}$, $C_1=\frac{3+r}{8\epsilon}$.}
\end{corollary}

%$\mathbf{Proof}$$\quad$
\begin{proof}
For any $r\in(0,1)$, let
      \begin{align*}
      \theta:=\frac{3+r}{4},\quad
      \delta:=\frac{1}{2\theta}=\frac{2}{3+r}\in \left(\frac{1}{2}, \frac{2}{3}\right)\subset\left(0,1\right), \quad
      p:=4\delta\theta=2,\quad
       q:=4\theta(1-\delta)=1+r.
         \end{align*}
 %By direct calculation, we have $p\in(1,+\infty)$, $q\in(1,2)\subset(1,+\infty)$, and for any $\delta\in \left(\frac{1}{2},\frac{2}{3}\right)\subset\left(0,1\right)$, we deduce further that

 By direct calculation, we have
                          \begin{align*}
                             \delta\left(\frac{1}{q}+1-\frac{1}{p}\right)=\frac{1}{q}.
                            \end{align*}
Therefore, for $v \in W_{[0]}^{1,2}(0,1)$, by using the inequality~\eqref{9}, we have
                \begin{align*}
             \|v\|_{L^\infty(0,1)}^{2}\leq\frac{1}{\delta^{2\delta}}\|v_y\|_{L^p(0,1)}^{2\delta}\|v\|_{L^q(0,1)}^{2-2\delta},
                 \end{align*}
which implies that
           \begin{align}\label{5}
             \|v\|_{L^2(0,1)}^{2\theta}\leq\frac{1}{\delta^{2\delta\theta}}\|v_y\|_{L^p(0,1)}^{2\delta\theta}\|v\|_{L^q(0,1)}^{2\theta-2\delta\theta}.
       \end{align}

For any $r\in (0,1)$, substituting the values of $\theta, \delta, p, q$ into \eqref{5},  we obtain
          \begin{align*}
    \|v\|_{L^2(0,1)}^{\frac{3+r}{2}} \leq\frac{3+r}{2}\|v_y\|_{L^2(0,1)}\|v\|_{L^{1+r}(0,1)}^{\frac{1+r}{2}}.
           \end{align*}
Then, by using the Young's inequality with $\epsilon\in\mathbb{R}_{>0}$ (see {\cite[Appendix B.2, p.706]{Evans:2010}}), we get
                        \begin{align*}
                 \|v\|_{L^2(0,1)}^{\frac{3+r}{2}} \leq\frac{3+r}{2}\cdot \epsilon\|v_y\|_{L^2(0,1)}^2 +\frac{3+r}{2}\cdot \frac{1}{4\epsilon} \|v\|_{L^{1+r}(0,1)}^{ 1+r},
                                  %   =&\frac{(3+r)\epsilon}{2}\int_0^1|v_y(y)|^2\mathrm{d}y+\frac{3+r}{8\epsilon}\int_0^1|v(y)|^{1+r}\mathrm{d}y,
               \end{align*}
 which gives the interpolation inequality~\eqref{interpolation-estimate}.  %\hfill $\blacksquare$
  \end{proof}
% \subsubsection{Proof of Proposition~\ref{prop}}\label{Sec: proof of Prop.}
%$\mathbf{Proof}$$\quad$

With the aid of the  interpolation inequality~\eqref{interpolation-estimate}, we prove Theorem~\ref{prop} by verifying the conditions of  the FTISS Lyapunov theorem are all fulfilled.

 \begin{proof}[Proof of Theorem~\ref{prop}.]
Let $P:=I$. Then, \eqref{P-coervice} holds true with $\mu_1:=1$. We claim that  $V(w):=\langle Pw,w \rangle_{X}=\|w\|_{L^2(0,1)}^2$  is an FTISS-LF of system~\eqref{sys'}.

Indeed,  for any $t\in\mathbb{R}_{>0}, w\in D(A)$, and $f\in U_c$, by direct calculation and noting that $A^*=A$,  we have
                           { \begin{align}
                             \langle (PA+A^*P)w,w \rangle_X
                              =2\langle Aw,w \rangle_{L^2(0,1)}
                              =2   \langle w_{yy},w \rangle_{L^2(0,1)}
                              %=&2\mu \left(w(y,t)w_y(y,t)|^{y=1}_{y=0}-\int_0^1w_y(y,t)\mathrm{d}w\right)\notag\\
                              %=&-2\mu k_2w^2(1)-2\mu \int^1_0w^2_y(y)\mathrm{d}y\notag\\
                              =-2\langle w_{y},w_y \rangle_{L^2(0,1)}
                             = -2 \|w_y\|_{L^2(0,1)}^2, \label{t1}
                            \end{align}}
% Combining the Poincar\'{e}'s inequality \cite[Theorem 5.8.1]{Evans:2010},
 % \begin{align}\label{t0}
  % \langle A^*w,w \rangle_X+\langle w,Aw \rangle_X\leq -2\mu  \|w\|^2_X.
  %\end{align}
and
             \begin{align}
              \langle \Xi(t,w,f),w \rangle_{X}
                   = & \langle  -k|w |^{r-1}w +f , w \rangle_{L^2(0,1)} \notag\\
                   =&\langle  -k|w |^{r-1}w ,w \rangle_{L^2(0,1)} +\langle f ,w \rangle_{L^2(0,1)}\notag \\
                  =&-k\|w\|_{L^{1+r}(0,1)}^{ 1+r}+\langle f ,w \rangle_{L^2(0,1)} \notag\\
                 %  \end{align*}
%By using the H\"{o}lder's inequality, we have
%               \begin{align}
               % \langle \Xi(t,w,f),w \rangle_X
               \leq & -k\|w\|_{L^{1+r}(0,1)}^{ 1+r}+\|f\|_{L^2(0,1)}\|w\|_{L^2(0,1)}\notag\\
               \leq &-k\|w\|_{L^{1+r}(0,1)}^{ 1+r}+\|f\|_{U_{c}}\|w\|_{L^2(0,1)}.\label{t2}
                                           \end{align}
%where $\|f\|_{U_{c}}=\sup\limits_{t\in\mathbb{R}_{>0}}\left(\int_{0}^{1}f^2[t](y)\mathrm{d}y\right)^{\frac{1}{2}}$.

It follows from \eqref{t1} and \eqref{t2} that
                         \begin{align}
                   \langle (PA,A^*P)w,w\rangle_X +  2\langle \Xi(t,w,f),w \rangle_X
                             \leq -2  \|w_y\|_{L^2(0,1)}^2-2k\|w\|_{L^{1+r}(0,1)}^{ 1+r}
                             +2\|f\|_{U_{c}}\|w\|_{L^2(0,1)}.\label{t5}
                             \end{align}

In view of Corollary \ref{coro},    we get
                                                       \begin{align}
                                        -2k\|w\|_{L^{1+r}(0,1)}^{ 1+r}
                                       \leq-\frac{2k}{C_1}\|w\|_{L^2(0,1)}^{\frac{3+r}{2}}+ \frac{2kC_0}{C_1}\|w_y\|_{L^2(0,1)}^2, \forall \epsilon\in\mathbb{R}_{>0},\label{t3}
                                      \end{align}
where  $C_0:=\frac{(3+r)\epsilon}{2}$ and $C_1:=\frac{3+r}{8\epsilon}$ with $\epsilon\in \mathbb{R}_{>0}$ to be determined later.

%Because $-kC_0<0$, we can further get
 %                    \begin{align}\label{t3}
  %                   -k\left(\int_0^1|w|^2{\rm d}y\right)^{\frac{3+r}{4}}\geq(-k)C_1\int_0^1|w|^{1+r}{\rm d}y.
   %                  \end{align}
We infer from \eqref{t5} and \eqref{t3} that
            {  \begin{align*}
            \langle (PA+A^*P)w,w \rangle_X +  2\langle \Xi(t,w,f),w \rangle_X
             % \leq &-2\mu \|w_y\|_{L^2(0,1)}^2-\frac{2k}{C_1}\|w\|_{L^2(0,1)}^{\frac{3+r}{2}}\\
%                                   &+\frac{2kC_0}{C_1}\|w_y\|_{L^2(0,1)}^2+2\|f\|_{U_{c}}\|w\|_{L^2(0,1)}\\
                              \leq&\left(-2+\frac{2kC_0}{C_1}\right)\|w_y\|_{L^2(0,1)}^2 -\frac{2k}{C_1}\|w\|_{L^2(0,1)}^{\frac{3+r}{2}}
                               +2\|f\|_{U_{c}}\|w\|_{L^2(0,1)}\\
                               =&(-2 +8k\epsilon^2)\|w_y\|^2_X-\frac{16k\epsilon}{3+r}\|w\|^{\frac{3+r}{2}}_X +2\|f\|_{U_{c}}\|w\|_X.
         \end{align*}}
Note that {$k\in\mathbb{R}_{>0}$}. We choose an arbitrary constant $\epsilon\in \mathbb{R}_{>0}$ such that {$-2+8k\epsilon^2< 0$}, i.e.,
{\begin{align}
0<\epsilon<\frac{1}{2\sqrt{k}}   .\label{epsilon}
\end{align}}
Then, we obtain
                 \begin{align*}
                 \langle(PA+A^*P)w,w \rangle_X +  2\langle  \Xi(t,w,f),w \rangle_X
               \leq-\frac{16k\epsilon}{3+r}\|w\|^{\frac{3+r}{2}}_X+2\|f\|_{U_{c}}\|w\|_X,
                  \end{align*}
which shows that the inequality \eqref{ro} holds true with $b:=\frac{16k\epsilon}{3+r}>0, \tau:=\frac{3+r}{2}\in\left(1,2\right), c:=2, \mu_2:=1$, and $ \zeta(s):=s$ for $s\in\mathbb{R}_{\geq 0}$.

According to Theorem \ref{th2}, for system~\eqref{sys'} with solution $w$, we deduce that $V(w):=\langle w,w \rangle_{X}=\|w\|_{L^2(0,1)}^2$  is an FTISS-LF. Thus, system~\eqref{sys'} is FTISS.
           %   \hfill $\blacksquare$
        \end{proof}
            \begin{remark}\label{Rem-2} By virtue of Remark~\ref{Rem-1} and the arbitrariness of $\epsilon$ in \eqref{epsilon},  the settling time, denoted by $T_*$,  of the disturbance-free system~\eqref{sys'}, or, equivalently, system~\eqref{sys}, satisfies
        { \begin{align}
            T_*\leq & \lim_{\epsilon \to \frac{1}{2\sqrt{k}}} \frac{2\mu^{\frac{\tau}{2}}_2}{ (2-\tau)b}V^{1-\frac{\tau}{2}}(w_0)
             %&\lim_{l\to 0^+}\frac{2V^{\frac{2-\tau}{2}}(w_0)}{(2-\tau)(b-l)\mu^{-\frac{\tau}{2}}_2}\notag\\
           = \lim_{\epsilon \to \frac{1}{2\sqrt{k}}} \frac{3+r}{4k\epsilon(1-r)}V^{\frac{1-r}{4}}(w_0)
           =\frac{3+r}{2\sqrt{k}(1-r)}\|w_0\|_{L^2(0,1)}^{\frac{1-r}{2}}.\label{12}
            \end{align}}

            \end{remark}
 \subsection{Numerical results}\label{Simulation}
In simulations, we always set {$r=0.6$}, and $k=2$. The initial data and in-domain disturbances are given by
    \begin{align*}
    w_0(y) =A_1 \left(y+\frac{1}{2}\right)^{\frac{1}{2}}\sin\left(3\pi  y+ \frac{\pi}{2} \right)    \quad\text{and}\quad
    f(y,t) =A_2{\sin\left(y+12t+6\right)},
    \end{align*}
    respectively,
where $A_1\in\{5, 50\}$ and $ A_2\in\{0, 20, 40\}$  are used to
describe the amplitude of initial data and in-domain disturbances.
\begin{figure}[t!]
    \centering

    \begin{minipage}{0.49\linewidth}
        \centering
        \includegraphics[width=0.9\linewidth]{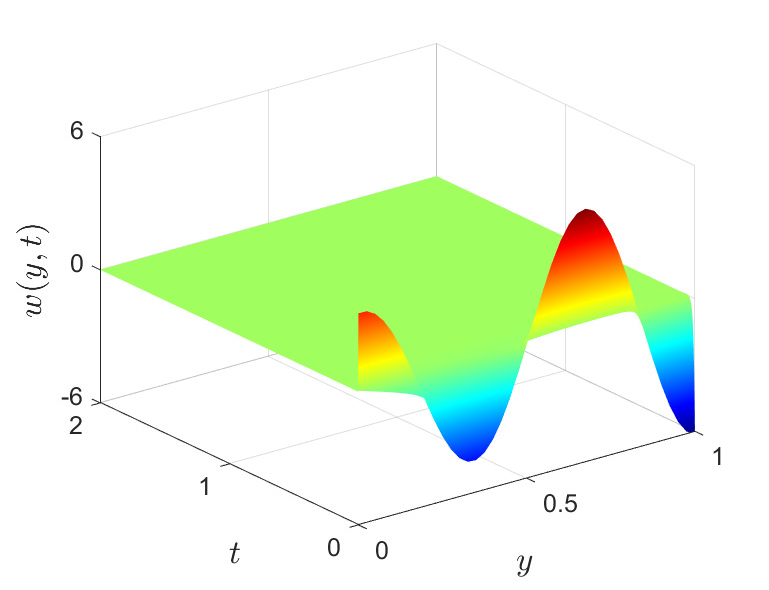}\\
        \centerline{(a) Evolution of $w$ for system~\eqref{sys} with $A_1=5$ and $A_2=0$.}
    \end{minipage}

    \begin{minipage}{0.49\linewidth}
        \centering
        \includegraphics[width=0.9\linewidth]{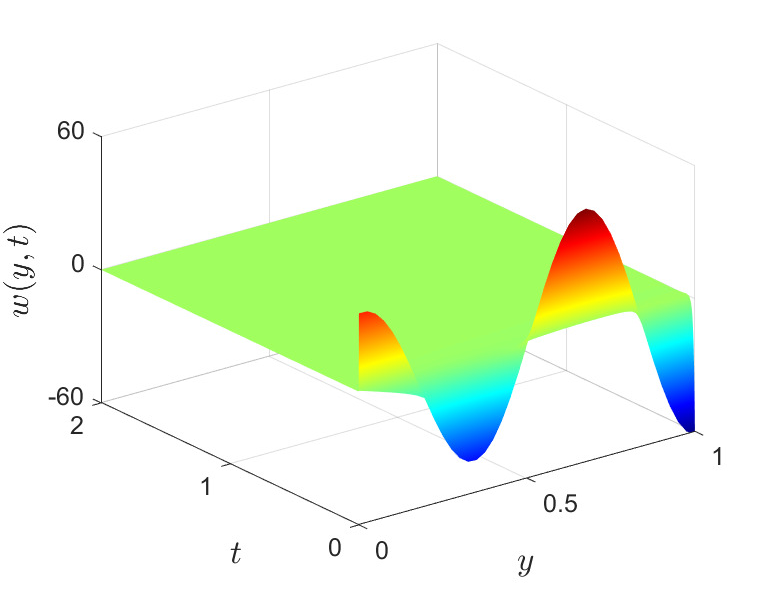}\\
        \centerline{(b) Evolution of $w$ for system~\eqref{sys} with $A_1=50$ and $A_2=0$.}
    \end{minipage}

    \begin{minipage}{0.49\linewidth}
        \centering
        \includegraphics[width=0.9\linewidth]{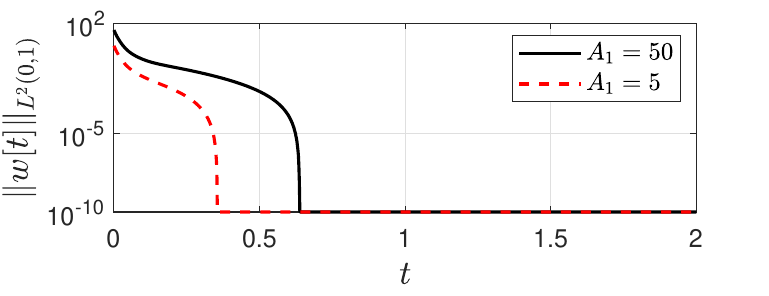}\\
        \centerline{(c) Evolution of $\|w[t]\|_{L^2(0,1)}$ for system~\eqref{sys} with $A_1\in\{5,50\}$ and $ A_2=0$.}
    \end{minipage}
    \captionsetup{justification=centering} % 设置 caption 居中
    \caption{\fontsize{10pt}{12pt}\selectfont Evolution of $w$ and $\|w[t]\|_{L^2(0,1)}$ for system~\eqref{sys} with different initial data.}
    \label{fig1}
\end{figure}
\begin{figure}[htbp!]
    \centering

    \begin{minipage}{0.49\linewidth}
        \centering
        \includegraphics[width=0.9\linewidth]{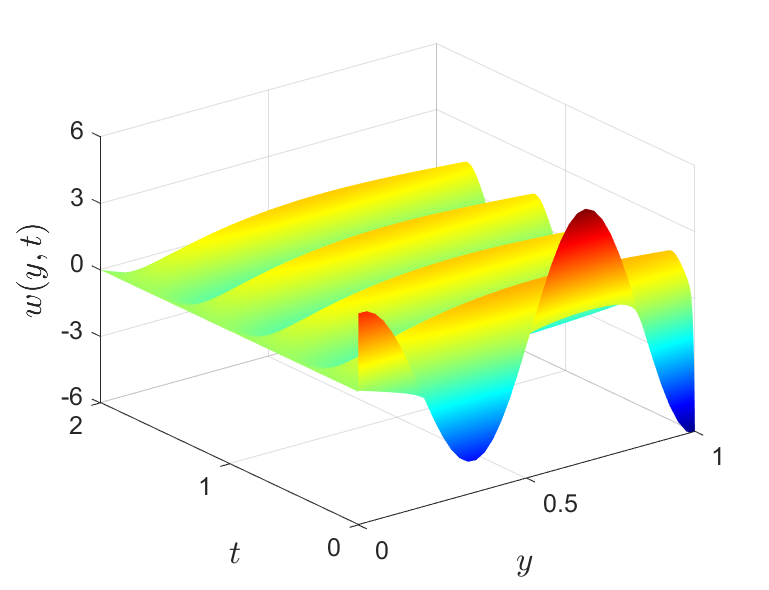}\\
        \centerline{(a) Evolution of $w$ for system~\eqref{sys} with $A_1=5$ and $A_2=20$.}
    \end{minipage}

    \begin{minipage}{0.49\linewidth}
        \centering
        \includegraphics[width=0.9\linewidth]{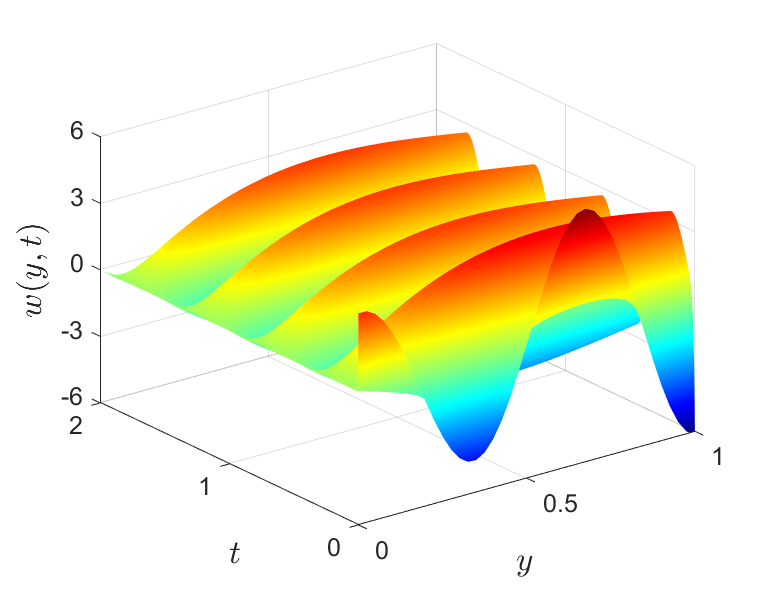}\\
        \centerline{(b) Evolution of $w$ for system~\eqref{sys} with $A_1=5$ and $A_2=40$.}
    \end{minipage}

    \begin{minipage}{0.49\linewidth}
        \centering
        \includegraphics[width=0.9\linewidth]{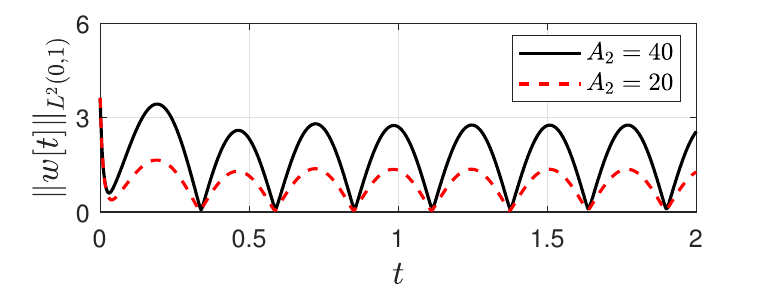}\\
        \centerline{(c) Evolution of $\|w[t]\|_{L^2(0,1)}$ for system~\eqref{sys} with $A_1=5$ and $ A_2\in \{20, 40\}$.}
    \end{minipage}

    \captionsetup{justification=centering} % 设置 caption 居中
    \caption{\fontsize{10pt}{12pt}\selectfont Evolution of $w$ and $\|w[t]\|_{L^2(0,1)}$ for system~\eqref{sys} with different disturbances.}\label{fig2}
\end{figure}

 In the absence of disturbances, namely, in the case where $A_2=0$, {Theorem~\ref{prop}} indicates that the disturbance-free system~\eqref{sys}   is  finite-time stable. This property is illustrated in  Fig.~\ref{fig1} (a), (b), and (c).
% Especially, it is shown  that for different initial data the disturbance-free system~\eqref{sys}   is always finite-time stable.
  Especially,  Fig.~\ref{fig1}~(c), which is plotted in a logarithmic scale,    well depicts the fast convergence  property  of solutions to the disturbance-free system~\eqref{sys}   with different initial data, i.e., $A_1\in\{5,50\}$. In addition, it can be seen from Fig.~\ref{fig1}~(c) that  the settling time decreases when the amplitude of initial data  decreases. This property is in accordance with  the theoretical result described by \eqref{12}.

In the presence of disturbances, namely, in the case where $  A_2\in\{  20, 40\}$, for the same initial data, i.e.,  $A_1=5$,  it is shown in Fig.~\ref{fig2} (a), (b), and (c) that the solutions of the disturbed system~\eqref{sys}  with different disturbances remain  bounded. Especially, the amplitude of solutions and their norms decreases when the amplitude of disturbances deceases. These robust properties, along  with the finite-time stability  property depicted by Fig.~\ref{fig2} (a) and (c), well illustrate  the FTISS of system~\eqref{sys}.

\section{conclusion}\label{s6}
 In this paper, we extended the notion of  FTISS    to infinite-dimensional systems and provide Lyapunov theory-based tools to establish the FTISS of infinite-dimensional systems. In particular, we demonstrated the construction of Lyapunov functionals tailored for assessing the FTISS for a class of  infinite-dimensional nonlinear systems under the framework of compact semigroup theory and Hilbert spaces, and verified the FTISS  for  parabolic PDEs with sublinear terms by using an interpolation inequality. Numerical results were presented to illustrate the obtained theoretical results. It is worth mentioning that designing feedback controls to achieve the FTISS for parabolic PDEs remains a challenging subject that will be considered in our future work.

\end{document}